\documentclass[a4paper,11pt,reqno]{amsart}
\usepackage{amsmath,amsthm,amssymb}
\usepackage[body={16cm,23cm},centering]{geometry} 
\usepackage{latexsym}
\usepackage{stmaryrd}
\usepackage{graphicx}
\usepackage{subfigure}
\usepackage{overpic}
\usepackage{tikz}
\usetikzlibrary{arrows.meta, decorations.markings, calc, positioning, decorations.pathreplacing}
\usepackage{enumerate}
\usepackage{esint}
\usepackage{xcolor}
\usepackage{bm}
\usepackage[normalem]{ulem}	

\theoremstyle{plain}
\newtheorem{theorem}{Theorem}[section]
\newtheorem{lemma}{Lemma}[section]

\theoremstyle{remark}
\newtheorem{remark}{Remark}[section]
\theoremstyle{definition}
\newtheorem{definition}{Definition}[section]

\numberwithin{equation}{section}

\newcommand{\R}{\mathbb{R}} 
\newcommand{\Om}{\Omega}
 
\newcommand{\p}{\partial}
\renewcommand{\d}{\delta}
\newcommand{\e}{\varepsilon}

\newcommand{\dd}{\mathrm{d}}
\DeclareMathOperator{\cof}{cof}
\DeclareMathOperator{\imG}{im_G}
\DeclareMathOperator{\dist}{dist}
\DeclareMathOperator{\essinf}{essinf}
\DeclareMathOperator{\Det}{Det}
\DeclareMathOperator{\dive}{div}
\renewcommand{\vec}[1]{\bm#1}

\begin{document}
\title[]{A Lavrentiev phenomenon in the neo-Hookean model}
\author{Marco Barchiesi}
\author{Duvan Henao}
\author{Carlos Mora-Corral}
\author{R\'emy Rodiac}
\date{\today}

\address[Marco Barchiesi]{
Dipartimento di Matematica, Informatica e Geoscienze, Universit\`a degli Studi di Trieste,
Via Weiss 2 - 34128 Trieste, Italy.
}
\email{barchies@gmail.com}

\address[Duvan Henao]{Instituto de Ciencias de la Ingenier\'ia, Universidad de O'Higgins. Rancagua, Chile}
\email{duvan.henao@uoh.cl}

\address[Carlos Mora-Corral]{Departamento de Matem\'aticas, Universidad Aut\'onoma de Madrid,
28049 Madrid, Spain  and Instituto de Ciencias Matem\'aticas,
CSIC-UAM-UC3M-UCM, 28049 Madrid, Spain.
}
\email{carlos.mora@uam.es}

\address[R\'emy Rodiac]{Laboratoire J.A. Dieudonn\'e, Universit\'e C\^ote d'Azur, CNRS UNMR 7351,06108, Nice, France.}
\email{remy.rodiac@univcotedazur.fr}

\begin{abstract}
We exhibit a Lavrentiev gap phenomenon for the neo-Hookean energy in three-
dimensional nonlinear elasticity. More precisely, we construct boundary data for
which the infimum of the neo-Hookean energy over deformations satisfying a natural
regularity and invertibility condition is strictly larger than the infimum over the
weak $H^1$-closure of that class. The mechanism underlying the gap is a deformation
with a dipole-type singularity.
\end{abstract}

\maketitle

\section{Introduction}
\subsection{Overview of the problem and statement of the result}
The neo-Hookean model is one of the simplest models in nonlinear elasticity. In this model, 
a deformation  $\vec u: \Om \rightarrow \R^3$ of an elastic body in a reference 
configuration $\Om \subset \R^3$ is assumed to minimise an energy of the form
\begin{equation}\label{eq:neo_Hook_energy}
E(\vec u)=\int_\Om \left[ |D \vec u|^2+ H(|\det D \vec u|) \right] \dd\vec x,
\end{equation}
where \(H: (0, \infty) \rightarrow [0, \infty) \) is a suitable convex function such that 
\begin{align}\label{eq:explosiveH}
\lim_{t\rightarrow \infty} \frac{H(t)}{t}=\lim_{s \rightarrow 0} H(s)=\infty.
\end{align}

The main mathematical difficulty is to identify an admissible class for which the
minimisation of $E$ is both physically meaningful and analytically tractable. Physical admissibility requires the deformation to preserve orientation and exclude interpenetration of matter, so two natural mathematical restrictions for the deformation are that $\vec u$ is one-to-one almost everywhere (a.e.\@), and \(\det D \vec u>0\) a.e\@. As a boundary condition, we prescribe a bi-Lipschitz orientation-preserving map $\vec b:\Om\rightarrow \R^3$.
In fact, for technical convenience, we impose the Dirichlet condition in a strong form: we fix a smooth bounded
domain $\widetilde{\Om}$ compactly contained in $\Om$ and require that each admissible deformation agrees with 
$\vec b$ on the entire complement $\Om\setminus\widetilde{\Om}$, not merely on $\partial\Om$.
This leads naturally to the admissible class
\begin{equation*}\begin{split}
\mathcal{A}= \mathcal{A}(\Omega, \widetilde\Om, \vec b) := \{ \vec u \in H^1(\Om,\R^3) \colon  \vec u = \vec b \text{ in } \Om\setminus\widetilde\Om, \,
 \vec u \text{ is one-to-one a.e.}, &\\  \det D\vec u>0 \text{ a.e.},  \text{ and }  E(\vec u)<\infty &\}.
\end{split}\end{equation*}

However, due to the possibility of cavitation,  i.e., the formation of voids in the material, it was shown in \cite{BallMurat1984} that \(E\) is not lower semicontinuous in  \(\mathcal{A}\). Since the difficulty is the possible appearance of many cavities, one can try to add a term to penalize the formation of cavities in the energy \eqref{eq:neo_Hook_energy}, or to add a condition in the minimisation space to exclude cavitation. This is the approach followed by M\"uller and Spector \cite{MuSp95}, where they also introduced the nowadays well-known condition INV in order to rule out some pathological examples of cavitation. Condition INV, which is a stronger form of invertibility than the property of being one-to-one a.e., informally means that after the deformation, matter coming from any subregion $U$ remains 
enclosed by the image of $\partial U$ and matter coming from outside $U$ remains exterior to the region enclosed 
by the image of $\partial U$. An alternative approach was devised in \cite{HeMo10,HeMo11} where, instead of condition INV, the authors work with one-to-one a.e.\ maps but penalize 
the formation of cavities  by using the ``divergence identities''. These are a generalisation of the identity \(\Det D \vec u=\det D \vec u\), which holds in absence of cavitation, where \(\Det D \vec u :=\frac 13 \dive (\cof D\vec u^T \vec u)\) denotes the distributional Jacobian. Both approaches are successful in treating energies slightly more coercive than \eqref{eq:neo_Hook_energy}, for example energies like \eqref{eq:neo_Hook_energy} where \(|D \vec u|^2\) is replaced by \(|D \vec u|^p\) for some \(p>2\), but fail to provide minimisers for \eqref{eq:neo_Hook_energy}. The difficulty is that, although condition INV and the divergence identities are well defined in the limiting case \(p=2\) (see \cite{CoDeLe03,HeMo12}), they are not sequentially closed with respect to the weak \(H^1\) convergence. This was demonstrated by an example of Conti and De Lellis \cite[Theorem 6.1] {CoDeLe03}, see also \cite{BaHEMoRo_24PUB,Dolezalova_Hencl_Maly_2023,Bouchala_Hencl_Zhu_2024}. This lack of compactness prevents the application of the direct method of the calculus of variations.
\emph{The goal of this paper} is to show that this lack of compactness is also related to a \emph{Lavrentiev phenomenon}. 
More precisely, we investigate the minimisation of \eqref{eq:neo_Hook_energy} in the space
\begin{equation}\label{def:Ar}
 \mathcal{A}^r=\mathcal{A}^r(\Omega,\widetilde \Omega, \vec b) := \{ \vec u \in \mathcal{A} \colon \Om_{\vec b} = \imG(\vec u,\Om) \text{ a.e. and } 
\vec u^{-1}\in W^{1,1}(\Om_{\vec b},\R^3)\},
\end{equation}
where $\Om_{\vec b}:=\vec b(\Om)$ and $\imG(\vec u, \Om)$ is the geometric image defined in Definition \ref{def:geometric_image}.
Here, the superscript $r$ stands for \emph{regular}. Indeed, by using \cite[Lemma 2.10]{BaHEMoRo_24_SIAM} and \cite[Lemma 5.1]{BaHeMo17}, one can show that maps in \(\mathcal{A}^r\) are regular in the sense that they satisfy  condition INV and the divergence identities (and, hence, they do not open any cavity). Our main result is the following.

\begin{theorem}\label{th:main1}
Let \(\Omega=B(\vec 0,4)\), and assume that \(H\) is convex and satisfies \eqref{eq:explosiveH} and
\begin{equation}\label{eq:growth condition}
    H(t)\leq c\, t^{-\alpha}
    \quad \text{as } t\to 0^+
    \qquad \text{and} \qquad
   H(t)\leq c\, t^\beta
   \quad \text{as } t\to \infty,
\end{equation}
 for some $\alpha<\tfrac{1}{3}$, $\beta < \tfrac{3}{2}$, and $c>0$. Then, there exist a subdomain  \(\widetilde \Om \subset \subset \Omega\) and a bi-Lipschitz homeomorphism \(\vec b : \Omega \rightarrow \R^3\) such that 
\begin{equation}\label{eq:Lavrentiev}
\inf\{E(\vec u) \colon \vec u \in \mathcal{A}^r\} > \inf\{E(\vec u) \colon \vec u \in \overline{\mathcal{A}^r}\},
\end{equation}
where $\overline{\mathcal{A}^r}$ is the closure of $\mathcal{A}^r$ with respect to the weak convergence in $H^1$.
\end{theorem}

Several remarks are in order concerning this theorem. First, it is not known if the infimum on the right-hand side of \eqref{eq:Lavrentiev} is attained. This is because we do not know if \( \det D\vec u_n \rightharpoonup \det D\vec u\) weakly in \(L^1\) whenever \( (\vec u_n)_n\subset \mathcal{A}^r\) is a minimising sequence for \(E\). Nevertheless, with the same proof used for Theorem \ref{th:main1} we can provide slightly different Lavrentiev phenomena where the infimum in $\overline{\mathcal{A}^r}$ is attained; see Section \ref{sec:variants}. Indeed, it was shown in \cite[Theorem 5.3]{Dolezalova_Hencl_Molchanova_2024} and \cite[Theorem 4.3]{Kalayanamit_2025} that the infimum of \(E\) in the sequential weak closure in \(H^1\) of homeomorphisms in \(H^1\) satisfying \(\det D\vec u\geq 0\) a.e.\@, Lusin's condition (N) and agreeing with \(\vec b\) on \(\Omega \setminus \widetilde \Omega\) is attained. However, a slight modification of Theorem \ref{th:main1} (see Theorem \ref{th:main_diffeo}) shows that, under assumption \eqref{eq:growth condition}, the minimisers they have obtained must have some singularity. Hence, the classical strategy of the calculus of variations, which consists of finding a minimiser of the energy in a larger space and then showing that minimisers are more regular, cannot work in the presence of a Lavrentiev phenomenon. This is an argument for using a relaxed energy instead of the original energy \(E\); this strategy was developed in  \cite{BaHEMoRo_23PUB,BaHEMoRo_24PUB,BaHEMoRo_24_SIAM}. 

Second, we notice that if the convex function \(H\) in \eqref{eq:neo_Hook_energy} satisfies \eqref{eq:explosiveH} and if there exists \(c>0\) such that 
\begin{equation}\label{eq:more_explosiveH}
c^{-1} H(t)\leq H(2t)\leq c\, H(t), \quad H(t) \geq \frac{1}{t^2}, \text{ for all  } t\in (0,\infty),
\end{equation}
then a weak \(H^1\) limit of orientation preserving homeomorphisms in \(\mathcal{A}\) does satisfy condition INV; see \cite{Dolezalova_Hencl_Maly_2023}.
This suggests that, under the growth condition \eqref{eq:more_explosiveH} there is no Lavrentiev gap phenomenon for \(E\) as the one in Theorem \ref{th:main1}. 

Third, it was recently shown in \cite{Campbell_Dolezalova_Hencl_2025} that even less coercive energies than \eqref{eq:neo_Hook_energy} admit minimisers in the sequential weak closure of some Sobolev, orientation preserving homeomorphisms. These energies are of the form \(E(\vec u)=\int_\Omega \left[ |D \vec u|^p +H(\det D \vec u)\right]\) with \(1<p<2\). It would be interesting to know if a Lavrentiev gap phenomenon holds in that case. 

Fourth, we observe that there is a strong analogy between the problem of minimising the neo-Hookean energy among regular maps and the problem of minimising the Dirichlet energy for maps \(\vec u :\Omega \subset \R^3\rightarrow \mathbb{S}^2\) among continuous maps. This analogy was already observed in \cite{GiMoSo98I,GiMoSo98II} and shows that the lack of compactness of both problems is caused by maps with dipole-like singularities. From that point of view, our Theorem \ref{th:main1} can be seen as an analogue of \cite[Theorem B]{Hardt_Lin_1986}. We also refer to \cite{Hardt_Lin_1992} and \cite{Mazowiecka_Strzelecki_2017} for more on the Lavrentiev gap phenomenon for harmonic maps, in particular we can ask if Lavrentiev gap phenomena such as the one in Theorem \ref{th:main1} hold for many boundary data, since it is the case in the context of harmonic maps, as shown in  \cite{Mazowiecka_Strzelecki_2017}. 
In the next section we present several Lavrentiev gap phenomena discovered in nonlinear elasticity.

\subsection{Lavrentiev gap phenomena in elasticity}

In nonlinear elasticity, the Lavrentiev phenomenon was first observed by Ball~\cite{Ball1982} for genuinely cavitating deformations, i.e., with cavities that actually open a hole. 
This is not the mechanism underlying Theorem \ref{th:main1}, in which a cavity is formed and then filled from material coming from another cavity, so that $\Om_{\vec b} = \imG(\vec u,\Om)$ a.e.\@
From the analytical point of view, this is reflected in the fact that a true cavitation map cannot be approximated within $\mathcal{A}^r$ 
under a uniform bound on the neo-Hookean energy, because such a bound would force the Jacobians to be equi-integrable. 
By contrast, the dipole singularity relevant here does admit an approximation by maps in $\mathcal{A}^r$ with equibounded neo-Hookean energy. 

The question whether the Lavrentiev phenomenon can occur under growth conditions on the stored-energy density that imply
continuity of all finite-energy deformations was raised by Ball and Mizel~\cite{BaMi85} (see also Ball~\cite{Ball02}).
This question was answered affirmatively in two dimensions by Foss~\cite{Foss03} and
Foss, Hrusa and Mizel~\cite{FoHrMi03a,FoHrMi03b}, who constructed examples on a disk sector with a smooth,
polyconvex, frame-indifferent energy density $W$ satisfying $W(F)\ge c\,|F|^p-C$ for some $p>2$ and
$W(F)\to\infty$ as $\det F\to 0^+$.
In those examples, the Lavrentiev gap is driven by the local behaviour of almost minimisers near the
tip of the sector, interacting with prescribed boundary conditions. 

More recently, Almi, Kr\"omer and Molchanova~\cite{AlKrMo24} produced Lavrentiev gaps in 
dimensions $d=2$ and $d=3$ for a neo-Hookean-type density $W(F)=|F|^p+\gamma(\det F)^{-q}$ with $p>d$
and $1 < q < p/(p-2)$.
Their admissible class consists of orientation-preserving $W^{1,p}$-maps satisfying the
Ciarlet--Ne\v{c}as condition~\cite{CiNe87}, which (in the regime $p>d$, hence for continuous maps)
implies almost everywhere injectivity.
The gap in~\cite{AlKrMo24} separates $W^{1,p}$-deformations from $W^{1,\infty}$-deformations.
It is generated by a global self-contact mechanism: for a suitable choice of boundary conditions on a
disconnected domain, one can compress two cross-sections onto the same point (or line), creating
self-contact while keeping the energy arbitrarily small; by contrast, Lipschitz competitors satisfying
the Ciarlet--Ne\v{c}as condition are necessarily injective everywhere (as a consequence of
Reshetnyak's theorem for mappings of finite distortion) and must therefore ``go around'' the obstruction,
incurring a strictly positive energy cost.

The present paper exhibits a Lavrentiev phenomenon of a fundamentally different nature, distinguished
from the existing literature in two complementary respects:

\textbf{(i) Natural coercivity exponent and physical energy.}
We work in three dimensions with the standard neo-Hookean functional
\begin{equation*}
  E(\vec u)=\int_\Om \Bigl(|D\vec u|^2+H(|\det D\vec u|)\Bigr)\,\dd\vec x.
\end{equation*}
The coercivity exponent here is $p=2<d=3$, namely the physically natural regime in which Sobolev embedding does not yield continuity of finite-energy deformations.
By contrast, the examples in~\cite{AlKrMo24,Foss03,FoHrMi03a,FoHrMi03b} are constructed in the continuity regime $p>d$.

\smallskip
\textbf{(ii) Gap driven by inverse regularity and lack of weak closure.}
In our setting, the energy gap does not arise from comparing two regularity levels of the forward
deformation $\vec u$ within a fixed notion of injectivity.
Rather, it stems from the failure of the physically motivated class $\mathcal A^r$ to be
sequentially closed under weak $H^1$-convergence.
Accordingly, the gap separates $\mathcal A^r$ from its weak closure $\overline{\mathcal A^r}$ and is
governed by the regularity of the \emph{inverse} deformation (Sobolev versus $BV$), rather than an
additional regularity of $\vec u$ itself.

This should be contrasted with \cite{AlKrMo24}. There, the admissibility condition is stable under weak convergence, but the class remains broad enough to permit low-energy deep self-contact. 
Here, by contrast, the admissible class rules out interpenetration but loses weak compactness. 
Thus \cite{AlKrMo24} points to an admissible class that is too large, whereas our example points to one that is too small.

\subsection{Strategy of the proof and organisation of the paper}

The proof of Theorem \ref{th:main1} is based on results that appeared in the series of papers \cite{BaHEMoRo_23PUB,BaHEMoRo_24PUB,BaHEMoRo_24_SIAM} and on the use of the Conti--De Lellis dipole map and its regularised sequence. More precisely, we have shown in \cite{BaHEMoRo_23PUB,BaHEMoRo_24PUB,BaHEMoRo_24_SIAM} that any sequence of regular maps approximating the Conti--De Lellis map produces a jump of energy of  at least \(2\pi\) in the limit. By using this observation and by constructing a suitable interpolation between the Conti--De Lellis map and a regularised sequence we can construct a boundary data for which  we obtain the desired gap phenomenon. 

The paper is organised as follows. Section \ref{sec:preliminaries} recalls some results from \cite{BaHEMoRo_23PUB,BaHEMoRo_24PUB,BaHEMoRo_24_SIAM} and the basic properties of the Conti--De Lellis map. These ingredients are used in Section \ref{sec:Lavrentiev} to prove our main result, Theorem \ref{th:main1}. Section \ref{sec:variants} presents two variants in which the infimum over the
larger class is attained. The main technical ingredient in the proof of Theorem \ref{th:main1} is the construction of a boundary data obtained from the Conti--De Lellis map and its regularised sequence. That construction is deferred to Appendix \ref{app:proof theo approx},  after Appendices \ref{app:dipole full} and \ref{app:dipole approx} where the precise definitions of the Conti--De Lellis map and its regularised sequences are given.

\section{The relaxed energy and the Conti--De Lellis map}\label{sec:preliminaries}
\subsection{The relaxed energy}

In \cite{BaHEMoRo_23PUB,BaHEMoRo_24PUB,BaHEMoRo_24_SIAM} the authors have proposed an approach to the minimisation problem for $E$
based on relaxation, by providing a larger space $\mathcal{B}\supset\mathcal{A}^r$ that is \emph{compact} for sequences with 
equibounded energy. In order to define this space we first define the \emph{geometric image} of a map \( \vec u\in \mathcal{A}\). We recall that maps in \(H^1(\Omega,\R^3)\) are approximately differentiable a.e.; see, e.g., \cite[Theorem 6.4]{Evans_Gariepy_2015}.
We denote the approximate gradient by $\nabla \vec u$ and the distributional derivative by $D \vec u$.

\begin{definition}\label{def:Om0}
Let $\vec u$ be approximately differentiable a.e.\ and such that $\det D \vec u\neq 0$ a.e. We define $\Om_0$ as the set of $\vec x\in \Om$ for which the following are satisfied:
\begin{enumerate}
\item the approximate differential of $\vec u$ at $\vec x$ exists and equals $D \vec u(\vec x)$,
\item there exist $\vec w\in C^1(\R^3,\R^3)$ and a compact set $K \subset \Om$ of density $1$ at $\vec x$ such that $\vec u|_{K}=\vec w|_{K}$ and $\nabla \vec u|_{K}=D \vec w|_{K}$,
\item $\det \nabla \vec u(\vec x)\neq 0$.
\end{enumerate}
\end{definition}
From \cite[Theorem 3.1.8]{Federer69}, 
Rademacher's Theorem and Whitney's Theorem we infer that \( \Om_0\) is a set of full Lebesgue measure in \( \Om \), i.e., \( |\Om \setminus \Om_0|=0\). 

\begin{definition}\label{def:geometric_image}
For any measurable set $A$ of $\Om$, the geometric image of $A$ under an a.e.\ approximately differentiable map $\vec u$ is defined by  
\begin{equation*}
\imG(\vec u,A) : =\vec u(A\cap \Om_0) ,
\end{equation*}
with \(\Om_0\) as in Definition \ref{def:Om0}.
\end{definition}

The notion of geometric image is used in the definition of the space
\begin{equation*}\begin{split}
\mathcal{B}:= \{\vec u \in H^1(\Om,\R^3) \colon & \vec u = \vec b \text{ in } \Om\setminus\widetilde\Om, 
 \, \vec u \text{ is one-to-one a.e.}, \, \det D\vec u\neq 0 \text{ a.e.}, \\
&\Om_{\vec b} = \imG(\vec u,\Om) \text{ a.e.}, \, \vec u^{-1}\in BV(\Om_{\vec b},\R^3), \text{ and }  E(\vec u)<\infty\} .
\end{split}\end{equation*}
On this set \(\mathcal{B}\) it is possible to define a \emph{lower semicontinuous} energy $F$ extending $E$, namely
\begin{equation*}\label{eq:F}
F(\vec u):= E(\vec u)+2 \| D^s \vec u^{-1} \| \text{ for } \vec u \in \mathcal{B}.
\end{equation*}
Here $D^s \vec u^{-1}$ is the singular part of the distributional gradient of the inverse (which is a matrix-valued Radon measure), 
and $\|D^s \vec u^{-1}\|$ is its norm.
Then, by using the direct method of the calculus of variations, one can obtain that the energy $F$ admits 
a minimiser $\vec u$ on $\mathcal{B}$. In this way the existence problem of a minimiser for $E$ is reduced to showing 
that $\vec u$ belongs to $\mathcal{A}^r$. This is summarised in the following result.

\begin{theorem}\cite[Theorem 1.1]{BaHEMoRo_24_SIAM} \label{main theorem}
Let $\{\vec u_j\}_j$ be a sequence in $\mathcal{B}$ such that $\{F(\vec u_j)\}_j$ is equibounded.
Then there exists $\vec u\in\mathcal{B}$ such that, up to a subsequence,
$\vec u_j \rightharpoonup \vec u$ in $H^1(\Om,\R^3)$ and 
\begin{equation*}
\liminf_{j\to\infty} F(\vec u_j)\geq F(\vec u).
\end{equation*} 
In particular, the energy $F$ has a minimiser in $\mathcal{B}$. 
\end{theorem}

\subsection{The Conti--De Lellis map}

Since the singular map $\vec v$ of Conti--De Lellis plays a central role in our construction, 
we briefly describe it and provide pictures here; for the details we refer the reader to Appendix \ref{app:dipole full}. 
The Conti--De Lellis map is an axisymmetric map defined in $B(\vec 0,4)$ with boundary condition equal to the identity. 
We remark that $\vec v\in\mathcal{B}$ and $\det D\vec v>0$ a.e., but 
$\vec v\notin\mathcal{A}^r$.
Indeed, the third component of the inverse is not Sobolev, but it belongs to the class $SBV$ 
(special functions of bounded variation). 
The key feature of $\vec v$ is that it brings into contact two portions of the body that were initially 
at unit distance apart, namely, the half-balls
\begin{equation*}
a:=\{\vec x:\, x_1^2 +x_2^2 +x_3^2 <1,\, x_3 < 0\}
\quad \text{and}\quad 
e:= \{\vec x:\, x_1^2 +x_2^2 +(x_3-1)^2<1,\, x_3 >1\}
\end{equation*}
(see Figure~\ref{fig:regions})
are put  in contact with each other across the ``bubble''
\begin{equation*}
\Gamma:=\{(y_1,y_2,y_3):\, y_1^2 + y_2^2 + (y_3-\tfrac{1}{2})^2 = (\tfrac{1}{2})^2\},
\end{equation*}
which in turn comes entirely from only two singular points: the origin $\vec 0$ and $\vec 0'=(0,0,1)$.
At $\vec 0'$ a cavity is nucleated and subsequently filled by material originating from the half-ball $a$,
which passes through the origin. 
We refer to this structure as a \emph{dipole}.
The jump set of the third component of the inverse coincides with the sphere $\Gamma$, and the amplitude of 
the jump is given by the distance between the poles $\vec 0$ and~$\vec 0'$. 
Consequently $F(\vec v)=E(\vec v)+2\pi$.

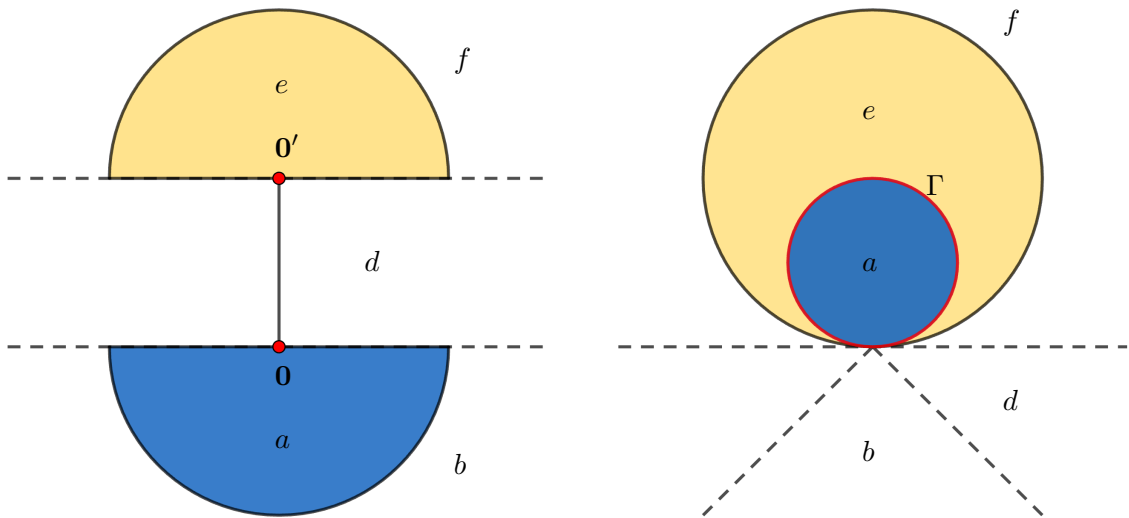
\begin{figure}
\centering
 \begin{overpic}[width=1.05\linewidth]{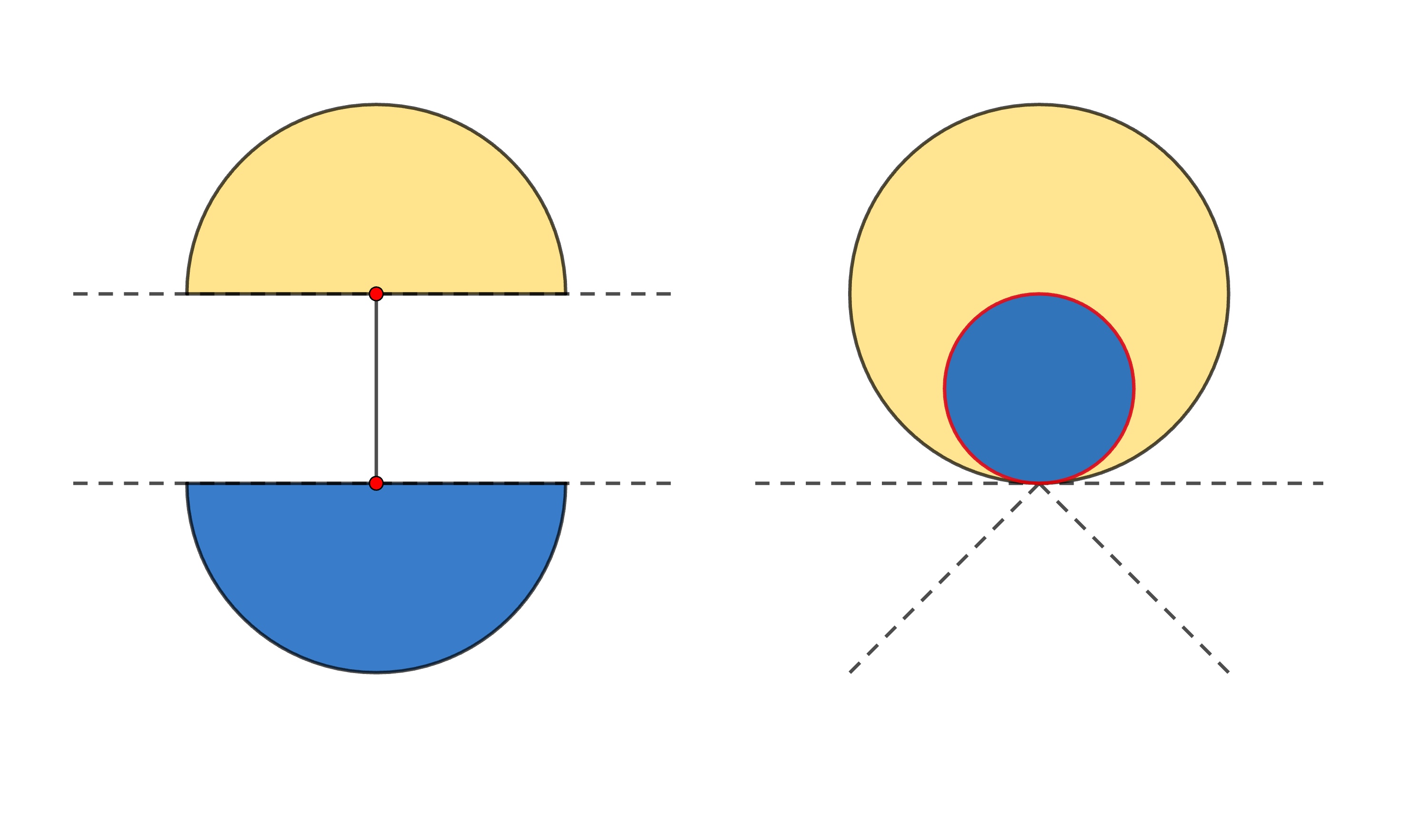}
\put (26,17) {$a$}
\put (40,15) {$b$}
\put (33,31) {$d$}
\put (26,45) {$e$}
\put (40,47) {$f$}
\put (26,22) {$\vec 0$}
\put (26,40) {$\vec 0'$}
\put (72,31) {$a$}
\put (72,16) {$b$}
\put (83,20) {$d$}
\put (72,43) {$e$}
\put (83,50) {$f$}
\put (77,37) {$\Gamma$}
\end{overpic}

\vspace*{-1cm}
\caption{Reference (left) and deformed (right) configurations ($2D$ sections) of the map $\vec v$ by Conti and De Lellis}
\label{fig:regions}
\end{figure}

It was shown in \cite[Theorem 1.2]{BaHEMoRo_24PUB} that  the dipole map $\vec v$ can be approximated by a sequence of maps $\vec u_\e\in\mathcal{A}^r$ from the energy point
of view, i.e., $\vec u_\e \rightharpoonup \vec v$ in $H^1$ and $\lim_{\e\rightarrow 0} E(\vec u_\e) = F(\vec v)$.
Roughly speaking, the maps $\vec u_\e$ are obtained from $\vec v$ by regularizing the jump of the inverse through a smooth junction (see the red layer in Figure~\ref{fig:regions-u_eps}). For the details we refer the reader to Appendix~\ref{app:dipole approx}.

\begin{figure}
\centering
\subfigure{
 \begin{overpic}[width=1.05\linewidth]{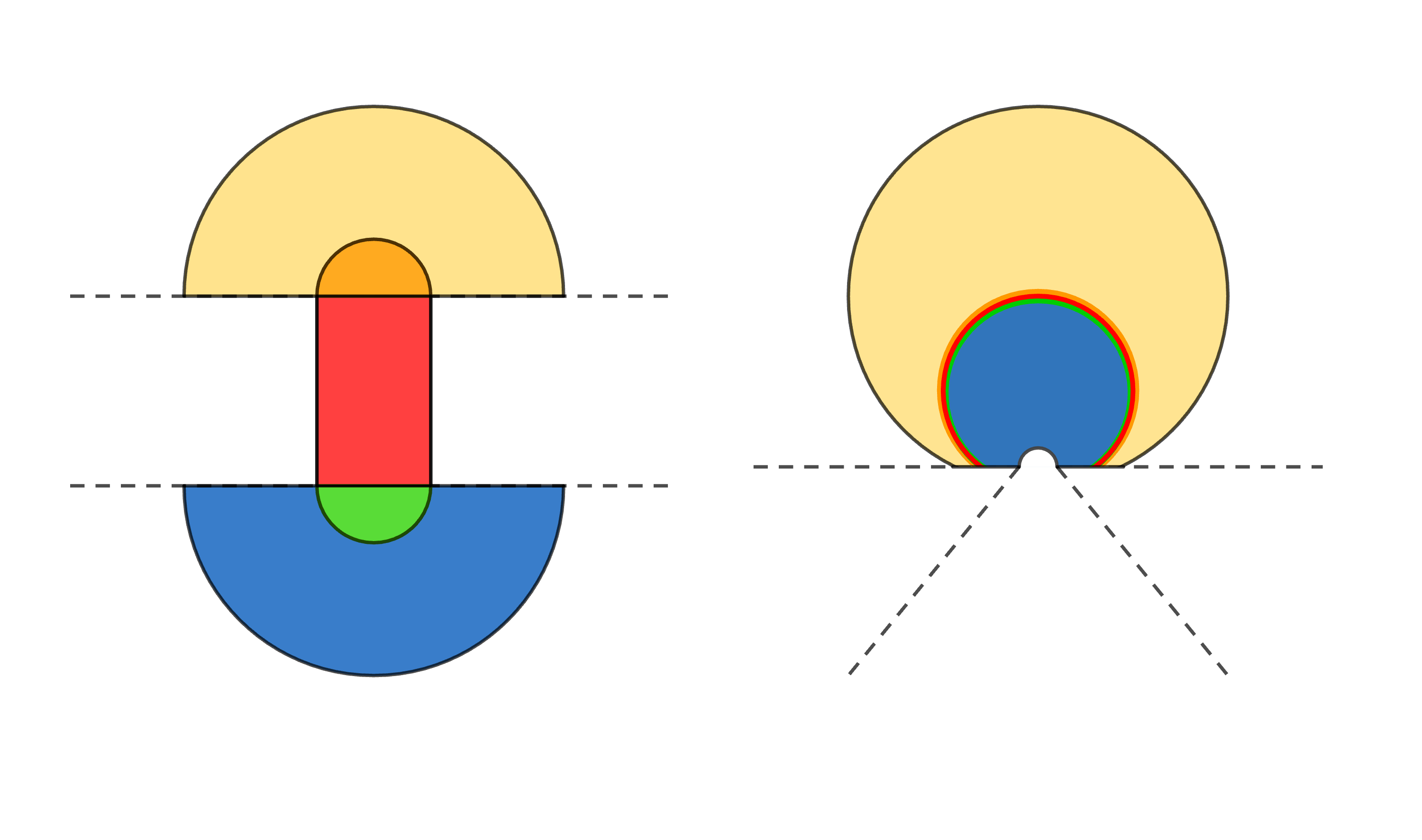}
\put (26,17) {$a_\e$}
\put (40,15) {$b$}
\put (33,31) {$d_\e$}
\put (26,45) {$e_\e$}
\put (40,47) {$f$}
\put (25,22) {$a_\e'$}
\put (25,40) {$e_\e'$}
\put (25,31) {$c_\e'$}
\put (72,31) {$a_\e$}
\put (72,16) {$b$}
\put (83,20) {$d_\e$}
\put (72,43) {$e_\e$}
\put (83,50) {$f$}
\end{overpic}}

\vspace*{-1cm}
\caption{Reference (left) and deformed (right) configurations ($2D$ sections) of the approximating map $\vec u_\e$}
\label{fig:regions-u_eps}
\end{figure}

\begin{figure}
\centering
\includegraphics[width=0.9\linewidth]{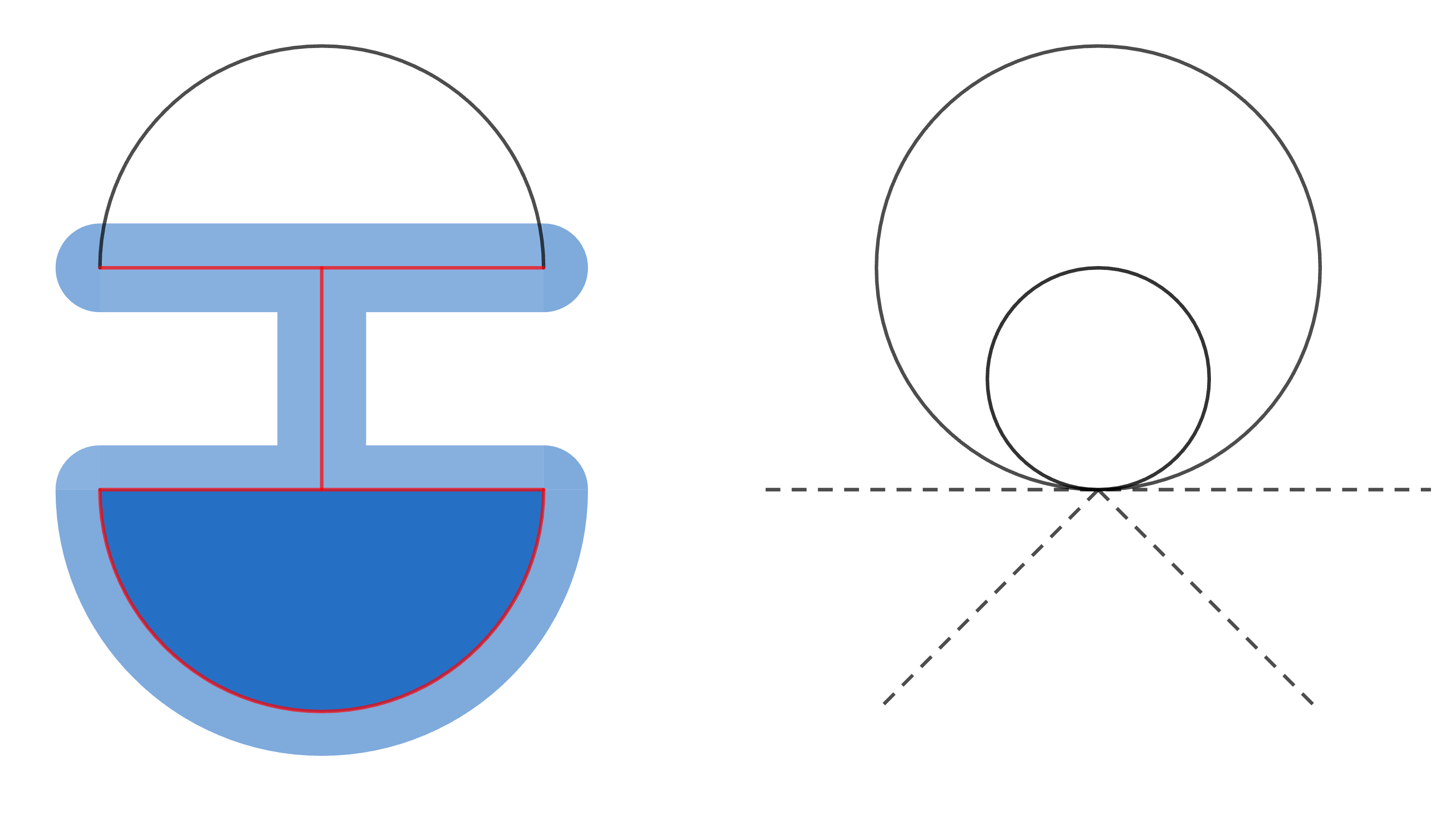}

\vspace*{-1cm}
\caption{In blue the $2D$ section of the region $U_\delta$}
\label{fig:regions-delta}
\end{figure}

The rough idea of the proof of Theorem \ref{th:main1} is to use \(\vec v\) as a boundary data to force approximate minimisers of the energy \(E\) to have the same singularities as \(\vec v\). However, the map $\vec v$ is not one-to-one: the segment $S:=\{(0,0,x_3) \colon x_3\in(0,1)\}$,
the punctured disks $D:=\{(x_1,x_2,0) \colon 0<x_1^2+x_2^2\leq1\}$ and $D':=\{(x_1,x_2,1) \colon 0<x_1^2+x_2^2\leq1\}$,
and the semisphere $\{\vec x:\, x_1^2 +x_2^2 +x_3^2 =1,\, x_3 < 0\}$
are contracted onto the origin. Nevertheless, outside the region 
\begin{align}\label{eq:U}
	U:=\{\vec x:\, x_1^2 +x_2^2 +x_3^2 \leq1,\, x_3 \leq0\}\cup S \cup D'
\end{align}
it is regular and it can be ``patched''. This means that we can construct a bi-Lipschitz homeomorphism from \(B (\vec 0,4)\) that agrees with \(\vec v\) outside a small neighbourhood of \(U\). This construction is the main ingredient in the proof of Theorem \ref{th:main1}.

\section{Proof of Theorem \ref{th:main1}}\label{sec:Lavrentiev}

We now proceed to the proof of Theorem \ref{th:main1} and start by a result containing the patching process alluded to in the previous section, where $\vec v$ is modified in the neighbourhood
\[
 U_\delta:=\{\vec x\in B(\vec 0,4)\colon \dist(\vec x, U)<\delta\} , \qquad \delta \in (0,1]
\]
of $U$ (see Figure ~\ref{fig:regions-delta}, in blue).

\begin{theorem}\label{theo: approx}
For each $\delta\in(0,1]$  there exists a bi-Lipschitz orientation-preserving homeomorphism $\vec b_\delta:B(\vec 0,4)\rightarrow \R^3$
such that $\vec b_\delta=\vec v$ on $B(\vec 0,4)\setminus U_\delta$. 
Moreover, if \(H\) is convex and satisfies \eqref{eq:explosiveH} and \eqref{eq:growth condition}
 for some $\alpha<\tfrac{1}{3}$, $\beta < \tfrac{3}{2}$, and $c>0$,
then
$\{E(\vec b_\delta)\}_{\delta>0}$ is equibounded.
\end{theorem}

Because the proof is technical,
we defer it to Appendix \ref{app:proof theo approx}. 

In what follows, when $\Om=B(\vec 0,4)$, $\widetilde{\Om}=U_\delta$ and $\vec b=\vec b_\delta$, we rename the family $\mathcal{A}^r=\mathcal{A}^r(\Omega, \widetilde \Omega, \vec b)$ as $\mathcal{A}^r_\delta$, where $\vec b_\delta$ is as in Theorem \ref{theo: approx}.
Moreover, let \(\overline{\mathcal{A}^r_\delta}\) denote the weak sequential closure of  \(\mathcal{A}^r_\delta\) in \(H^1\).
Because of the construction of the $\vec b_\delta$ in Theorem \ref{theo: approx}, we have that
\begin{equation}\label{eq:inclusionArd}
 \mathcal{A}^r_{\delta'} \subset \mathcal{A}^r_\delta \qquad \text{for } 0 < \delta' \leq \delta \leq 1 .
\end{equation}

In our example, $\mathcal{A}^r_\delta$ plays the role of the regular class and the underlying mechanism of the Lavrentiev phenomenon 
is the energy concentration of the dipole $\vec v$ at the two singular points, $\vec 0$ and $\vec 0'$. This amount of energy is precisely 
$2\pi=2\| D^s \vec v^{-1} \|$ and cannot be detected by $E$.

We present the main result of the article, of which Theorem \ref{th:main1} is a particular case.

\begin{theorem}\label{th:main_quantitative}
Assume that $H$ is a convex function that satisfies conditions \eqref{eq:explosiveH} and \eqref{eq:growth condition}. Then, for all \(\lambda \in (0,2\pi)\) there exists $\delta\in(0,1]$ such that 
\begin{equation*}
\inf\{E(\vec u) \colon \vec u \in \mathcal{A}^r_\delta\}\geq\inf\{E(\vec u) \colon \vec u \in \overline{\mathcal{A}^r_\delta}\}+\lambda.
\end{equation*}
\end{theorem}
\begin{proof}
Suppose, for the sake of contradiction, that the claim is false for some $\lambda \in (0, 2 \pi)$. Then, for any $\delta \in (0,1]$ we can select $\vec w_\delta\in\mathcal{A}^r_\delta$ such that
\begin{equation}\label{ineq1}
E(\vec w_\delta)<\inf\{E(\vec u) \colon \vec u \in \overline{\mathcal{A}^r_\delta}\}+\lambda.
\end{equation}
As $\vec b_\delta \in \mathcal{A}^r_\delta$, we can choose $\vec w_{\delta}$ such that $E(\vec w_\delta)\leq E(\vec b_\delta)$.
Then, by Theorem \ref{theo: approx}, the set $\{E(\vec w_\delta)\}_{\delta \in (0,1]}$ 
is bounded.
From \eqref{eq:inclusionArd} we obtain that $\vec w_\delta \in \mathcal{A}^r_1$.
Then, by Theorem \ref{main theorem},
there exist a sequence $\{ \delta_j \}_j$ in $(0,1]$ tending to zero and a map $\vec w\in\overline{\mathcal{A}^r_1}$ such that 
\begin{equation}\label{ineq2}
 \vec w_{\d_j} \rightharpoonup \vec w \text{ in } H^1(\Om,\R^3) , \qquad \vec w_{\d_j} \to \vec w \text{ a.e. \quad and \quad } \liminf_{j\to\infty} E(\vec w_{\d_j})\geq F(\vec w).
\end{equation} 
By \cite[Proposition 3.1]{BaHEMoRo_24_SIAM}, the inverse map $\vec w^{-1}$ belongs to $BV$.
Let us see that it fails to be in a Sobolev space by estimating $\| D^s \vec w^{-1} \|$ from below. Indeed, by \eqref{ineq2} and the boundary condition $\vec w_{\d_j} = \vec v$ in $\Om \setminus U_{\d_j}$, we have that $\vec w$ coincides with $\vec v$ a.e.\ outside $U$, and, hence, outside the half-ball $a$.
Therefore, $\vec w^{-1}$
coincides with $\vec v^{-1}$ a.e.\ outside the ball $B((0,0,1/2),1/2) = \vec v (a)$. In particular, \(\vec w^{-1}\) sends \( \vec v (e)\) to \( e\), so $\vec w^{-1}$ has also
a discontinuity on $\Gamma$ and the amplitude of the jump is at least $1$ (the distance of the two half-balls 
$a$ and $e$ in the reference configuration). Then $\| D^s \vec w^{-1} \|\geq\pi$ and by \eqref{ineq2}
\begin{equation}\label{ineq3}
\liminf_{j\to\infty} E(\vec w_{\d_j})\geq E(\vec w)+2\pi.
\end{equation} 
Using \eqref{eq:inclusionArd}, we have that $\vec w \in \overline{\mathcal{A}^r_\delta}$ for all $\delta \in (0, 1]$, so
\[
\inf\{E(\vec u) \colon \vec u \in \overline{\mathcal{A}^r_\delta}\}\leq E(\vec w)
\]
and, in particular,
\begin{equation}\label{ineq4}
 \liminf_{j \to \infty} \inf\{E(\vec u) \colon \vec u \in \overline{\mathcal{A}^r_{\delta_j}}\}\leq E(\vec w) .
\end{equation}
From \eqref{ineq1} we obtain
\begin{equation}\label{ineq5}
 \liminf_{j \to \infty} E(\vec w_{\delta_j}) \leq \liminf_{j \to \infty} \inf\{E(\vec u) \colon \vec u \in \overline{\mathcal{A}^r_{\delta_j}}\}+\lambda.
\end{equation}
Combining \eqref{ineq3}, \eqref{ineq4} and \eqref{ineq5} we get a contradiction.
\end{proof}

\section{Variants of Theorem \ref{th:main1}}\label{sec:variants}

In this section we give two variants of our main Theorem \ref{th:main1}, and of its  version in Theorem \ref{th:main_quantitative}, in which the infimum in the larger space is attained.

\subsection{The axisymmetric setting}
We first give a variant of Theorem \ref{th:main1} in the axisymmetric setting. For \(\Omega\) and \(\widetilde \Omega\) two smooth axisymmetric bounded open subsets of \(\R^3\) such that \(\widetilde \Omega  \subset \subset \Omega\) and for \(\vec b: \widetilde \Omega \rightarrow \R^3\) a bi-Lipschitz axisymmetric homeomorphism, we define
\begin{equation}
\mathcal{A}_s =\mathcal{A}_s(\Omega, \widetilde \Omega, \vec b):=\{ \vec u \in \mathcal{A} \colon  \vec u \text{ is axisymmetric} \}
\quad\text{ and }\quad
\mathcal{A}_s^r := \mathcal{A}_s\cap \mathcal{A}^r.
\end{equation}
We denote by \(\overline{\mathcal{A}}_s^r\) the weak sequential closure in \(H^1\) of maps in \(\mathcal{A}_s^r\). For the precise definition of axisymmetric maps we refer to \cite[Section 2.3]{BaHEMoRo_23PUB}.
The variant of Theorems \ref{th:main1} and \ref{th:main_quantitative} in this context is as follows.

\begin{theorem}\label{th:main_axi}
We assume that $H$ is a convex function that satisfies conditions \eqref{eq:explosiveH} and \eqref{eq:growth condition}. Set \(\Omega=B(\vec 0,4)\).
Then for all \(\lambda\in (0,2\pi)\) there exist an axisymmetric domain \(\widetilde \Omega \subset \subset \Omega\) and a bi-Lipschitz axisymmetric homeomorphism \(\vec b : \Omega \rightarrow \R^3\) such that 
\begin{align}
\inf \{ E(\vec u): \vec u \in \mathcal{A}_s^r\} & \geq \min \{ E(\vec u): \vec u \in  \overline{\mathcal{A}}_s^r \} +\lambda \label{RHS1} \\
& \geq \min \{ E(\vec u): \vec u \in \mathcal{A}_s\}+\lambda. \label{RHS2}
\end{align}
\end{theorem}

\begin{proof}
The proof is exactly the same as the one of Theorem \ref{th:main_quantitative}, with the same choices of \(\widetilde \Omega=U_\delta\) and \(\vec b=\vec b_\delta\). The key point is that the Conti--De Lellis map is axisymmetric and so is its patch obtained in Theorem \ref{theo: approx}.
\end{proof}
We note that the minima on the right-hand side of \eqref{RHS1} and \eqref{RHS2}  are attained thanks to \cite[Proposition 6.6 and Theorem 3.2]{BaHEMoRo_23PUB}. We also observe that \(\overline{\mathcal{A}}_s^r  \subset \mathcal{A}_s\) from \cite[Theorem 5.1]{BaHEMoRo_23PUB}.

\subsection{Working with the weak closure of Sobolev homeomorphisms}

In another direction, in a general non-axisymmetric setting, instead of working with \(\mathcal{A}^r\) and its weak closure we can work with the class of Sobolev homeomorphisms and its weak closure.  For \(\Omega\) and \(\widetilde \Omega\) two smooth bounded open subsets of \(\R^3\) such that \(\widetilde \Omega  \subset \subset \Omega\) and for \(\vec b: \widetilde \Omega \rightarrow \R^3\) a bi-Lipschitz  homeomorphism we define
\begin{equation*}\begin{split}
\mathcal{H}=\mathcal{H}(\Omega,\widetilde \Omega, \vec b) :=\{ \vec u: \Omega \rightarrow \vec b(\Omega) & \text{ homeomorphism satisfying Lusin's condition (N),} \\
& \,\,\,\vec u=\vec b \text{ on } \Omega \setminus \widetilde \Omega \text { and } E(\vec u)\leq E(\vec b) \}.
\end{split}\end{equation*}
Recall that a map \(\vec u:\Omega \rightarrow \R^3\) satisfies Lusin's condition (N) if \(|\vec u(A)|=0\) for any \(A\subset \R^3\) such that \(|A|=0\).
We denote by \(\overline{\mathcal{H}}\) the weak closure in \(H^1\) of maps in \(\mathcal{H}(\Omega,\widetilde \Omega, \vec b)\).
The variant of Theorems \ref{th:main1} and \ref{th:main_quantitative} in this context is as follows.

\begin{theorem}\label{th:main_diffeo}
We assume that $H$ is a convex function that satisfies conditions \eqref{eq:explosiveH} and \eqref{eq:growth condition}. Let \(\Omega=B(\vec 0,4)\).
Then for all \(\lambda \in (0,2\pi)\) there exist a smooth open set \(\widetilde \Omega \subset \subset \Omega\) and a bi-Lipschitz  homeomorphism \(\vec b:\Omega\setminus \widetilde \Omega \rightarrow \R^3\)  such that 
\begin{align}
\inf \{ E(\vec u): \vec u \in \mathcal{H}\} \geq \min \{ E(\vec u): \vec u \in  \overline{\mathcal{H}} \} +\lambda. \label{eq:RHS3}
\end{align}
\end{theorem}

\begin{proof}
The proof is again the same as in Theorem \ref{th:main_quantitative}. The key points are that homeomorphisms satisfying Lusin's condition (N) in turn satisfy the divergence identities (see \cite[Theorem 5.4]{HeMo12}) and, hence, \(\mathcal{H} \subset \mathcal{A}^r\), and that the patch of the Conti--De Lellis map \(\vec u_\delta\) obtained in Theorem \ref{theo: approx} is a bi-Lipschitz homeomorphism (and hence satisfies Lusin's condition). 
\end{proof}
Note that the minimum in the right-hand side of \eqref{eq:RHS3} is attained due to \cite[Theorem 5.3]{Dolezalova_Hencl_Molchanova_2024} (see also \cite[Theorem 4.3]{Kalayanamit_2025}).


\appendix

\section{The dipole map}\label{app:dipole full}

The definition of the limiting map in the Conti--De Lellis example is constructed by partitioning the ball \( B(\vec 0,3)\) into several distinct subregions. 
By axisymmetry, it suffices to describe \(\vec v\) in the right half-plane.
We describe \(\vec v\) by using spherical coordinates \( (v_\rho, v_{\theta} =\theta, v_\varphi)\)
in the image.  Note that, when \( \theta=0\), the vector \(\vec e_r \) equals \( \vec e_1\); as a consequence, \( (v_\rho,v_\varphi)\) are the polar coordinates of the map \( \vec v\) restricted to the plane generated by \( (\vec e_1, \vec e_3)\). 

\smallskip
\noindent\textbf{Region} $a:=\{ \rho \sin \varphi \vec e_r(\theta)+\rho \cos \varphi \vec e_3 
: \ 0\leq \rho \leq 1, \frac{\pi}{2} \leq \varphi \leq \pi \}$. 
In this region we set
\begin{gather*}
    \vec v \big ( \rho \sin\varphi \vec e_r(\theta) +\rho\cos\varphi\vec e_3 \big ) 
    = v_\rho \sin v_\varphi\,\vec e_r(\theta) + v_\rho \cos v_\varphi\, \vec e_3,
    \\
    v_\rho(\rho, \varphi) = (1-\rho)\cos v_\varphi, \qquad
    v_\varphi(\rho,\varphi) = \pi-\varphi.
\end{gather*}

\smallskip
\noindent\textbf{Region} $b:=\{ \rho \sin \varphi \vec e_r(\theta)+\rho \cos \varphi \vec e_3 : 1<\rho\leq 3, \frac{\pi}{2}\leq \varphi \leq \pi\}$. 
Here we define \(\vec v \) by
\begin{gather*}
	\vec v\big (\rho\sin\varphi\,\vec e_r(\theta) + \rho\cos\varphi\,\vec e_3\big )= v_\rho \sin v_\varphi\,\vec e_r(\theta) + v_\rho \cos v_\varphi\, \vec e_3,
    \\
    v_\rho(\rho, \varphi) = \rho-1, \qquad
    v_\varphi(\rho,\varphi) = \frac{\varphi+\pi}{2}. 
\end{gather*}

\smallskip
\noindent\textbf{Region} $e:=\{ \rho \sin \varphi \vec e_r(\theta) +(1+\rho \cos \varphi) \vec e_3 
: 0\leq \rho \leq 1, 0 \leq \varphi \leq \frac{\pi}{2}\}$.
Here we set
	\begin{gather*}
    \vec v\big ( \vec e_3 + \rho\sin\varphi \vec e_r(\theta) + \rho\cos\varphi \vec e_3\big ) := v_\rho  \sin\varphi\, \vec e_r(\theta) + v_\rho \cos \varphi\,\vec e_3,\\
    v_\rho( \rho, \varphi) := (1+\rho)\cos \varphi, \quad  v_\varphi=\varphi.
    \end{gather*}
    
\smallskip
\noindent\textbf{Region} $f:=\{ \vec e_3 + \rho \sin \varphi \vec e_r(\theta)+\rho \cos \varphi \vec e_3 
: \rho \geq 1, 0\leq \varphi \leq \frac{\pi}{2}\}\cap B(\vec 0,3)$.
   In the original construction provided by \cite{CoDeLe03},
    the only requirements are
    that:
    \begin{itemize}
        \item The 2D representative of
    $$
    	\vec v\big (\vec e_3 + \rho\sin\varphi\vec e_r(\theta) + \rho\cos\varphi \vec e_3\big ) = v_\rho\sin v_\varphi \vec e_r(\theta) + v_\rho\cos v_\varphi\vec e_3, 
    	\quad \rho\geq 1,\ \varphi \in [0,\frac{\pi}{2}], 
    $$
    sends the 2D region 
    \begin{align*}
        \{(\rho\sin\varphi, 1+ \rho\cos\varphi): \rho\geq 1, 0\leq \varphi\leq \frac{\pi}{2}\}
        \cap B\big ( (0,0),3\big )
    \end{align*}
    in a bi-Lipschitz manner onto its image, which is contained in 
    \begin{align*}
        \{(v_\rho\sin v_\varphi, v_\rho\cos v_\varphi): v_\rho\geq 2\cos v_\varphi,\  0\leq  v_\varphi\leq \frac{\pi}{2}\}.
    \end{align*}

    \item  The map $\vec v$ matches the definition of region $e$ at the interface:
        $$
        v_\rho(1,\varphi):=2\cos\varphi,
        \qquad 
    	v_\varphi(1,\varphi):=\varphi.
    $$
    In particular,  $v_\rho(1,\frac{\pi}{2})=0$.
    \medskip

    \item The map $\vec v$ sends the half-line $\{(\rho, \frac{\pi}{2}): \rho\geq 1\}$ to the ray $\{\varphi=\tfrac{\pi}{2}\}$:
    $$
        v_\varphi(\rho, \frac{\pi}{2}) = \frac{\pi}{2},
        \qquad \rho\geq 1.
    $$
    \end{itemize}

Here, we simply use 
     \begin{align}
            \label{eq:particular_f}
         v_\rho(\rho,\varphi) = 2\cos\varphi + (\rho-1),
         \qquad
         v_\varphi(\rho,\varphi) = \varphi.
     \end{align}

\bigskip
\noindent\textbf{Region} $d:= \{ 0 \leq x_3\leq 1\}\cap B(\vec 0,3)$. 
In \cite{CoDeLe03}
the limit map $\vec v$ is defined as the composition of two maps. 
    The first one is an auxiliary generic axisymmetric
    \begin{align}\label{eq:auxiliary g}
         \vec g   \big ( \vec x\big)
        = 
        s(r,x_3)\,\vec e_r(\theta) + z(r,x_3)\,\vec e_3,
        \qquad 
        \vec x = r\vec e_r  +x_3 \vec e_3,
    \end{align}
    whose planar representative 
    is a bi-Lipschitz transformation
    from the 2D region
    $$
    \{(r,x_3):\ r\geq 0,\ 0\leq x_3\leq 1,\ r^2 +x_3^2 < 3^2\}
    $$
    onto its image, which is contained in 
    $$
    \{(s, z):\ s\geq 0,\ 0\leq z\leq 3\}.
    $$
    This transformation $\vec g$
    is then composed with the
    transformation from the cylindrical coordinates $(s,z)$ to the spherical coordinates $(s,\varphi(z))$:
    \begin{equation}\label{eq:def-v-region-d}    	\vec v \big ( r\vec e_r(\theta) + x_3\,\vec e_3\big ) 
    	= s\,\sin \big (\varphi(z)\big)\,\vec e_r(\theta) 
    	- s\,\cos \big ( \varphi(z)\big ) \,\vec e_3,
    \end{equation}
    with 
    $$
    \varphi(z):= \frac{\pi}{4}\bigg ( 1 + \frac{z}{3} \bigg ).$$
    As a result, the region $d$
    is mapped to the sector
    $\tfrac{\pi}{4}\leq \varphi\leq \tfrac{\pi}{2}$ in spherical coordinates
    (see Figure~\ref{fig:regions}).

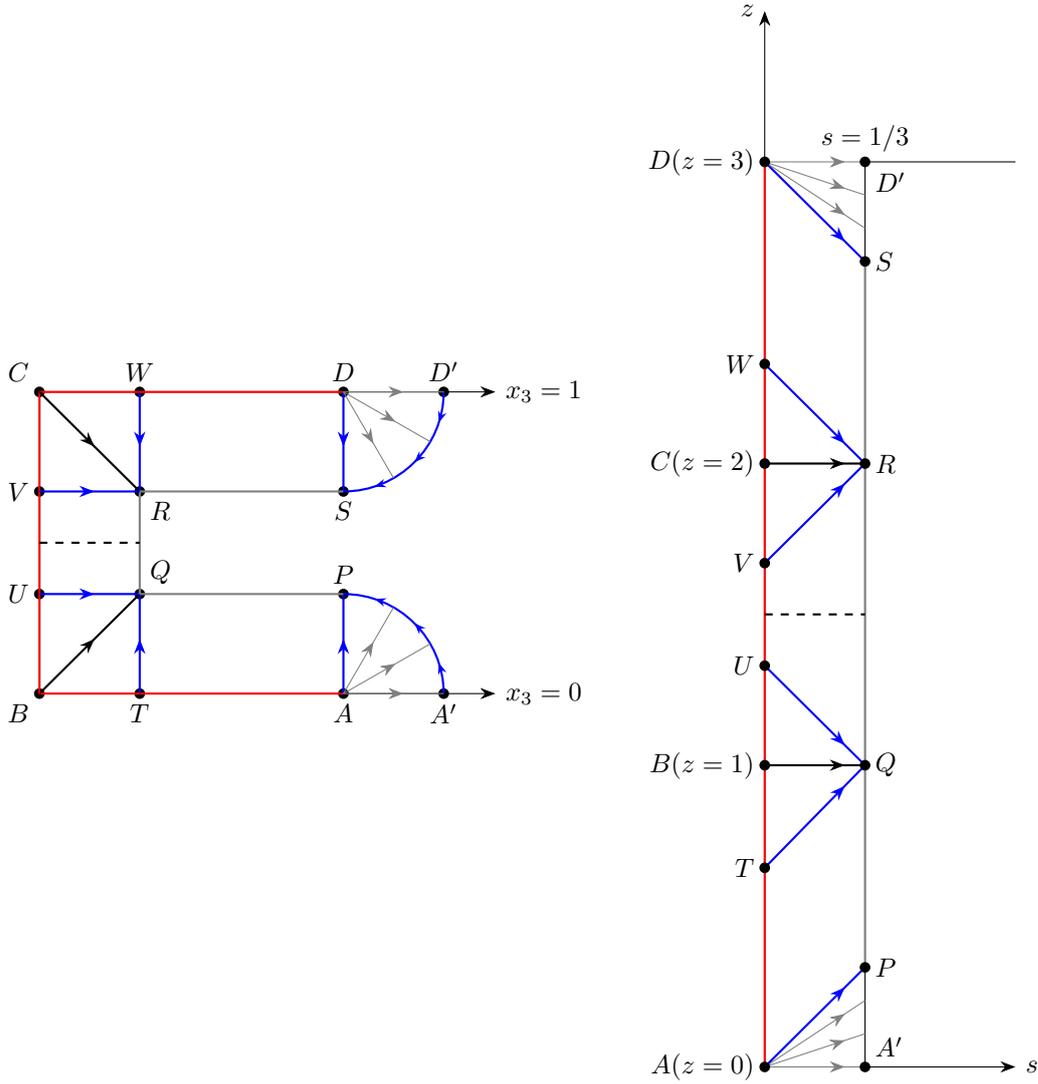
\begin{figure}[ht]
\centering

\begin{minipage}[c]{0.49\textwidth}
\centering
\begin{tikzpicture}[
    scale=4,
    >={Stealth[length=2mm]},
    flow/.style={postaction={decorate, decoration={markings, mark=at position 0.55 with {\arrow{>}}}}},
    bflow/.style={postaction={decorate, decoration={markings, mark=at position 0.6 with {\arrow{>}}}}},
    curveflow/.style={
        decoration={
            markings,
            mark=at position 0.2 with {\arrow{Stealth[length=1.5mm]}},
            mark=at position 0.5 with {\arrow{Stealth[length=1.5mm]}},
            mark=at position 0.8 with {\arrow{Stealth[length=1.5mm]}}
        },
        postaction={decorate}
    },
    font=\small
]

\coordinate (A) at (1,0);
\coordinate (Aprime) at (1.33,0);
\coordinate (B) at (0,0);
\coordinate (C) at (0,1);
\coordinate (D) at (1,1);
\coordinate (Dprime) at (1.33,1);

\coordinate (T) at (0.33, 0);
\coordinate (U) at (0, 0.33);
\coordinate (V) at (0, 0.67);
\coordinate (W) at (0.33, 1);

\coordinate (P) at (1, 0.33);
\coordinate (Q) at (0.33, 0.33);
\coordinate (R) at (0.33, 0.67);
\coordinate (S) at (1, 0.67);

\def\pointradius{0.5pt}
\fill (A) circle (\pointradius) node[below] {$A$};
\fill (B) circle (\pointradius) node[below left] {$B$};
\fill (C) circle (\pointradius) node[above left] {$C$};
\fill (D) circle (\pointradius) node[above] {$D$};

\fill (T) circle (\pointradius) node[below] {$T$};
\fill (U) circle (\pointradius) node[left] {$U$};
\fill (V) circle (\pointradius) node[left] {$V$};
\fill (W) circle (\pointradius) node[above] {$W$};

\fill (Q) circle (\pointradius) node[above right] {$Q$};
\fill (R) circle (\pointradius) node[below right] {$R$};
\fill (P) circle (\pointradius) node[above] {$P$};
\fill (S) circle (\pointradius) node[below] {$S$};

\fill (Aprime) circle (\pointradius) node[below] {$A'$};
\fill (Dprime) circle (\pointradius) node[above] {$D'$};

\draw[->] (0,0) -- (1.5,0) node[right] {$x_3=0$}; 
\draw[->] (0,1) -- (1.5,1) node[right] {$x_3=1$};
\draw[thick] (B) -- (C); 

\draw[thick, red] (A) -- (T);
\draw[thick, blue, flow] (A) -- (P);
\draw[thick, blue, flow] (T) -- (Q);
\draw[thick, gray] (Q) -- (P); 
\draw[thick, red] (T) -- (B);
\draw[thick, black, flow] (B) -- (Q);
\draw[thick, red] (B) -- (U);
\draw[thick, blue, flow] (U) -- (Q);
\draw[thick, red] (U) -- (V);
\draw[thick, gray] (R) -- (Q);
\draw[thick, red] (V) -- (C);
\draw[thick, red] (C) -- (W);
\draw[thick, blue, flow] (V) -- (R);
\draw[thick, black, flow] (C) -- (R);
\draw[thick, blue, flow] (W) -- (R);
\draw[thick, red] (W) -- (D);
\draw[thick, blue, flow] (D) -- (S);
\draw[thick, gray] (R) -- (S);

\coordinate (IUV) at ($(U)!0.5!(V)$); 
\coordinate (IQR) at ($(Q)!0.5!(R)$); 
\draw[thick, dashed] (IUV) -- (IQR);

\draw[thick, blue, curveflow] (Dprime) arc (0:-90:0.33);
\draw[thick, blue, curveflow] (Aprime) arc (0:90:0.33);

\coordinate (P1) at ($(A)+({0.33*cos(60)},{0.33*sin(60)})$);
\coordinate (P2) at ($(A)+({0.33*cos(30)},{0.33*sin(30)})$);
\draw[gray,bflow] (A) -- (P1);
\draw[gray,bflow] (A) -- (P2);
\draw[gray,bflow] (A) -- (Aprime);

\coordinate (S1) at ($(D)+({0.33*cos(60)},{-0.33*sin(60)})$);
\coordinate (S2) at ($(D)+({0.33*cos(30)},{-0.33*sin(30)})$);
\draw[gray,bflow] (D) -- (S1);
\draw[gray,bflow] (D) -- (S2);
\draw[gray,bflow] (D) -- (Dprime);

\end{tikzpicture}
\end{minipage}
\hspace{-25pt}
\begin{minipage}[c]{0.49\textwidth}
\centering
\begin{tikzpicture}[
    scale=4,
    >={Stealth[length=2mm]},
    flow/.style={postaction={decorate, decoration={markings, mark=at position 0.8 with {\arrow{>}}}}},
    font=\small
]

\coordinate (Origin) at (0,0);
\coordinate (A_def) at (0,0);
\coordinate (T_def) at (0,0.66);
\coordinate (B_def) at (0,1);
\coordinate (U_def) at (0,1.33); 
\coordinate (V_def) at (0,1.67);
\coordinate (IUV) at ($(U_def)!0.5!(V_def)$); 
\coordinate (C_def) at (0,2);
\coordinate (W_def) at (0,2.33);
\coordinate (D_def) at (0,3);
\def\sval{0.33}
\coordinate (P_def) at (\sval, 0.33);
\coordinate (P1) at (\sval, 0.22);
\coordinate (P2) at (\sval, 0.11);
\coordinate (Aprime) at (\sval, 0.0);
\coordinate (Q_def) at (\sval, 1);
\coordinate (R_def) at (\sval, 2);
\coordinate (IQR) at ($(Q_def)!0.5!(R_def)$); 
\coordinate (S_def) at (\sval, 2.67);
\coordinate (S1) at (\sval, 2.78);
\coordinate (S2) at (\sval, 2.89);
\coordinate (Dprime) at (\sval, 3.0);
\draw[->] (0,0) -- (0,3.5) node[left] {$z$};
\def\endrightaxisdeformed{2.5*\sval}
\draw[->] (0,0) -- (\endrightaxisdeformed,0) node[right] {$s$};
\draw[thin, black] (\sval,0) -- (\sval,3.0) node[above] {$s=1/3$};

\draw[gray, flow] (A_def) -- (P1); 
\draw[gray, flow] (A_def) -- (P2); 
\draw[gray, flow] (A_def) -- (Aprime); 
\draw[thick, blue, flow] (A_def) -- (P_def); \draw[thick, gray] (P_def) -- (Q_def);
\draw[thick, blue, flow] (T_def) -- (Q_def); 
\draw[thick, red] (A_def) -- (T_def);
\draw[thick, red] (T_def) -- (B_def);
\draw[thick, black, flow] (B_def) -- (Q_def); 
\draw[thick, red] (B_def) -- (U_def);
\draw[thick, blue, flow] (U_def) -- (Q_def);
\draw[thick, gray] (Q_def) -- (R_def); \draw[thick, blue, flow] (V_def) -- (R_def);
\draw[thick, red] (U_def) -- (V_def);
\draw[thick, dashed] (IUV) -- (IQR);
\draw[thick, red] (V_def) -- (C_def);
\draw[thick, black, flow] (C_def) -- (R_def);
\draw[thick, red] (C_def) -- (W_def);
\draw[thick, blue, flow] (W_def) -- (R_def);
\draw[thick, gray] (R_def) -- (S_def);
\draw[thick, blue, flow] (D_def) -- (S_def);
\draw[thick, red] (W_def) -- (D_def);
\draw[gray, flow] (D_def) -- (S1); 
\draw[gray, flow] (D_def) -- (S2); 
\draw[gray, flow] (D_def) -- (Dprime); 
\draw[black] (Dprime) -- (\endrightaxisdeformed,3.0);

\def\pointradius{0.5pt}
\fill (A_def) circle (\pointradius) node[left] {$A (z=0)$};
\fill (T_def) circle (\pointradius) node[left] {$T$};
\fill (B_def) circle (\pointradius) node[left] {$B (z=1)$};
\fill (U_def) circle (\pointradius) node[left] {$U$};
\fill (V_def) circle (\pointradius) node[left] {$V$};
\fill (C_def) circle (\pointradius) node[left] {$C (z=2)$};
\fill (W_def) circle (\pointradius) node[left] {$W$};
\fill (D_def) circle (\pointradius) node[left] {$D (z=3)$};
\fill (P_def) circle (\pointradius) node[right] {$P$};
\fill (Q_def) circle (\pointradius) node[right] {$Q$};
\fill (R_def) circle (\pointradius) node[right] {$R$};
\fill (S_def) circle (\pointradius) node[right] {$S$};
\fill (Aprime) circle (\pointradius) node[above right] {$A'$};
\fill (Dprime) circle (\pointradius) node[below right] {$D'$};


\end{tikzpicture}
\end{minipage}

\caption{The transformation $\vec g = s(r,x_3)\vec e_r + z(r,x_3)\vec e_3$ from the (planar representative slice of) region $d$ in the reference configuration (left) onto the intermediate  $(s,z)$ configuration (right).}
\label{fig:function g}
\end{figure}

    The transformation $\vec g$ must fulfill the following requirements:
    \begin{itemize}
        \item The points $$A(r=1, x_3=0),\quad  B(r=0, x_3=0),\quad  C(r=0, x_3=1),\quad D(r=1, x_3=1),$$
  respectively, are mapped to 
$$(s=0,z=0),\quad (s=0,z=1),\quad (s=0,z=2),\quad (s=0,z=3).$$ 
In the segments joining those points,  $\vec g$ is affine.
(After composing with the polar transformation $s\sin\varphi\,\vec e_r -s\cos\varphi\,\vec e_3$, the polygonal curve $ABCD$ is contracted to a point.)
    \smallskip 
    
     \item 
     At the interface $\{r\geq 1,\ x_3=1\}$,
	\begin{equation}
        \label{eq:continuity_f_d}
  	 s(r,1):=v_\rho(r,\frac{\pi}{2}),
     \quad z(r,1) \equiv 3,
	\end{equation}
    where $v_\rho(r,\varphi)$ is the radial spherical coordinate of the image in region $f$.
    In this article we choose a particular definition of this function, namely, the one in \eqref{eq:particular_f}. 
    Hence, \eqref{eq:continuity_f_d}
    becomes 
    \[
  	 s(r,1):=r-1, \quad z(r,1) \equiv 3.
	\]
    \medskip 
    
	\item 
    At the interface $\{r\geq 1, x_3=0\}$
	\begin{equation*}
		s(r ,0):=r-1,\quad z(r,0)\equiv 0.
	\end{equation*}
    This is consistent with the definition of the dipole in region $b$.
    \end{itemize}
    \bigskip

    In \cite{CoDeLe03,BaHEMoRo_24PUB} there is flexibility regarding the way to define $\vec g$ in the interior of the slab.    
    Here we are  going to use a specific function $\vec g$ (see Figure \ref{fig:function g}) whose radial coordinate $s(r,x_3)$ is given exactly by 
    \begin{align}
        \label{eq:def_g_specific}
        s=\dist(\vec x, U), \quad \vec x=r\vec e_r+x_3\vec e_3
    \end{align}
    in the subregion 
    $\{\vec x: 0\leq \dist(\vec x, U)\leq \tfrac{1}{3}\}$.

\bigskip
\noindent\textbf{Bounds on the Jacobian}.
Throughout the rest of the paper, $c$ will denote a universal positive constant whose precise value may change from  line to line.
We will also use the following elementary result.
\begin{lemma}
If $H : (0, \infty) \to \R$ is convex and $\lim_{s \to 0} H(s) = \infty$ then there exists $\delta>0$ such that $H$ is decreasing in $(0, \delta)$. 
\end{lemma}
\begin{proof}
As $H$ is convex, its right derivative $H'_+$ exists everywhere on $(0,\infty)$ and is
non-decreasing.
Then the limit $\lim_{s \to 0} H'_+(s)$ exists, and must be $- \infty$ since $\lim_{s \to 0} H(s) = \infty$.
Therefore, $H'_+$ is negative in $(0,\delta)$ for some $\delta>0$.
\end{proof}

In region $d$ we have\footnote{See, e.g., \cite[Appendix B.4]{BaHEMoRo_24PUB} for the computation of the determinant of a map expressed in cylindrical-spherical coordinates.}
\begin{align*}
    \det D\vec v(\vec x) = \frac{s^2(r,x_3)\sin\varphi(z(r,x_3))}{r}
    \geq \frac{\sqrt{2}}{2}\frac{s^2}{r}
\end{align*}
since $\tfrac{\pi}{4}\leq \varphi \leq \tfrac{\pi}{2}$.
This determinant goes to zero as $\vec x$ approaches 
the polygonal line $ABCD$ on $\partial U$.
Focusing on the region where $s\leq \tfrac{1}{3}$
(see Figure \ref{fig:function g}),
we have that $r\leq \tfrac{4}{3}$ and, hence,
$\det D\vec v \geq c \, s^2$.
By symmetry with respect to the plane $\{x_3=\tfrac{1}{2}\}$, it is enough to consider
$\vec x=r\vec e_r+x_3\vec e_3$ with $x_3\in [\tfrac{1}{2},1]$.
Here one has
\[
\dist(\vec x,U)\geq \min\{r,\,1-x_3\},
\]
since the distance from $\vec x$ to the axis is $r$, while the distance from $\vec x$
to the top disk $\{x_3=1,\ 0\leq r\leq 1\}$ is at least $1-x_3$.
Hence, by \eqref{eq:def_g_specific}, we have
\[
\det D\vec v(\vec x)
\geq c\,s^2(r,x_3)
\geq c\,\min\{r^2,\,(1-x_3)^2\}.
\]
Using that $H$ is decreasing near $0$, and using \eqref{eq:growth condition}, we obtain
\[
H(\det D\vec v(\vec x))
\leq c\,\min\{r^2,\,(1-x_3)^2\}^{-\alpha}
\leq c\Bigl(r^{-2\alpha}+(1-x_3)^{-2\alpha}\Bigr).
\]
Therefore,
\begin{equation}\label{eq:integrability_H}
\int_{\{s\leq \tfrac{1}{3}\}} H(\det D\vec v)\,\dd\vec x
\leq
c\int_{0}^{4/3}\int_{1/2}^{1}
\Bigl(r^{-2\alpha}+(1-x_3)^{-2\alpha}\Bigr)\,2\pi r\,\dd x_3\dd r,
\end{equation}
i.e., in this region $H(\det D\vec v)$ is integrable if $\alpha<\frac{1}{2}$.

\smallskip
In region~$e$ we have $v_\varphi=\varphi$ and
$v_\rho(\rho,\varphi)=(1+\rho)\cos\varphi$,
so that $\p_\rho v_\rho=\cos\varphi$ and\footnote{See, e.g., \cite[Appendix B.3]{BaHEMoRo_24PUB} for the computation of the determinant of a map expressed in spherical-spherical coordinates.}
\[
  \det D\vec v
  =\frac{v_\rho^2}{\rho^2}\,\p_\rho v_\rho
  =\frac{(1+\rho)^2\cos^3\varphi}{\rho^2}\,.
\]
This determinant is close to zero near the interface with region~$d$
(particularly at the triple junction, where $\rho\approx 1$ and
$\cos\varphi\approx 0$) and unbounded near the cavitation point
$\vec 0'=(0,0,1)$, where $\rho\to 0^+$.
To estimate the integral of $H(\det D\vec v)$ we split region~$e$
into $\{\det D\vec v\le 1\}$ and $\{\det D\vec v>1\}$.
On the first set the growth condition~\eqref{eq:growth condition}
gives $H(t)\le c\,t^{-\alpha}$ for all $t\in(0,1]$
(after adjusting the constant~$c$, since $H$ is continuous on $(0,\infty)$),
while on the second set $H(t)\le c\,t^{\beta}$ for all $t\ge 1$.
Writing the volume element in spherical coordinates as
$\dd\vec x=2\pi\,\rho^2\sin\varphi\,\dd\rho\,\dd\varphi$
and enlarging each integral to the full region~$e$, we obtain
\begin{align*}
  \int_{\{\det D\vec v\le 1\}\cap\, e} H(\det D\vec v)\,\dd\vec x
  &\le
  c\int_0^{\pi/2}\!\!\int_0^1
  \Bigl(\frac{(1+\rho)^2\cos^3\varphi}{\rho^2}\Bigr)^{\!-\alpha}
  \rho^2\sin\varphi\;\dd\rho\,\dd\varphi\\
  &= c
  \int_0^{\pi/2}\cos^{-3\alpha}\!\varphi\;
    \sin\varphi\;\dd\varphi
  \;\cdot\;
  \int_0^1(1+\rho)^{-2\alpha}\rho^{2+2\alpha}\dd\rho\,,
\end{align*}
and
\begin{align*}
  \int_{\{\det D\vec v > 1\}\cap\, e} H(\det D\vec v)\,\dd\vec x
  &\le
  c\int_0^{\pi/2}\!\!\int_0^1
  \Bigl(\frac{(1+\rho)^2\cos^3\varphi}{\rho^2}\Bigr)^{\!\beta}
  \rho^2\sin\varphi\;\dd\rho\,\dd\varphi\\
  &= c
  \int_0^{\pi/2}\cos^{3\beta}\!\varphi\;
    \sin\varphi\;\dd\varphi
  \;\cdot\;
  \int_0^1(1+\rho)^{2\beta}\rho^{2-2\beta}\dd\rho\,.
\end{align*}
Both right-hand sides are finite under condition~\eqref{eq:growth condition} with \(\alpha<1/3\) and \(\beta<3/2\),
so $H(\det D\vec v)$ is integrable in region~$e$.

One can check that in the remaining regions the integrability of the
Jacobian is not worse (for more details, see \cite[Section 3]{BaHEMoRo_24PUB}).

\bigskip

\section{The approximating sequence}\label{app:dipole approx}

We recall from \cite[Theorem 1.2]{BaHEMoRo_24PUB} that for every sufficiently small $\e>0$ there exists 
an axisymmetric map $\vec u_\e\in\mathcal{A}^r$ with the following properties:
\begin{enumerate}[(i)]
  \item $\vec u_\e$ is a bi-Lipschitz homeomorphism of $B(\vec0,4)$, orientation preserving, and coincides with the identity on $\p B(\vec0,4)$;
  \item $\vec u_\e\rightharpoonup \vec v$ in $H^1(B(\vec0,4),\R^3)$ and a.e.\ in $B(\vec0,4)$;
  \item the energies are uniformly bounded, i.e.,
  \begin{equation*}
    \sup_{\e>0}E(\vec u_\e)<\infty.
  \end{equation*}
\end{enumerate}

The approximation $\vec u_\e$ is defined (in the regions of our interest) as follows. 
We refer to \cite{BaHEMoRo_24PUB} for the full construction.
From now on, a choice is made of
	$$
		0<\gamma\leq \frac{1}{3},
		\quad \text{a fixed positive exponent.}
	$$
This range for $\gamma$ played an important role in \cite{BaHEMoRo_24PUB}
in order to show that  $\vec u_\e$
is an approximating sequence for the dipole $\vec v$  with optimal energy
(compare with \eqref{eq:defEta}).


\smallskip
\noindent\textbf{Region} $b$. We work with spherical coordinates 
\( (\rho,\theta,\varphi)\), \( 1<\rho <3, \ 0\leq \theta <2 \pi, \ \pi/2\leq \varphi \leq \pi\) 
in the reference configuration and cylindrical coordinates in the deformed configuration. We first define 
the auxiliary functions 
\begin{equation*}\begin{split}
\vec d(\varphi,\theta)&:=\sin \left(\frac{\varphi + \pi}{2} \right)\,\vec e_r(\theta)+\cos \left(\frac{\varphi + \pi}{2} \right)\,\vec e_3,\\[5pt]
\vec\phi_\e(\vec x)&:= (\rho-1+\sqrt2\,\e^\gamma)\,\vec d(\varphi,\theta)+\e^\gamma\,\vec e_3.
\end{split}\end{equation*}

With this notation, we have $\vec v(\vec x)= (\rho-1)\,\vec d(\varphi,\theta)$.
The factor $\sqrt{2}$ appears because, according to the definition of $\vec u_\e$ in region $a_\e$, the image of $A$ is
$(\e^\gamma, 0)$, whose distance to $(0,\e^\gamma)$ is $\sqrt{2}\e^\gamma$. In region $b$ we define $\vec u_\e$ to be 
\begin{align} \label{eq:def-u_eps-region-b}
\vec u_\e (\vec x) = \e^\gamma \vec \psi \Big (\e^{-\gamma} \vec \phi_\e(\vec x) \Big ),
\end{align}
where $\vec\psi$ is any axisymmetric bi-Lipschitz bijection from
$$
	\Big \{(x_1, x_2, x_3):\ r^2 + (x_3-1)^2 \geq 2,\ x_3\leq 0,\ r\leq 1 + |x_3|\Big \}
$$
onto
$$
	\Big \{(x_1, x_2, x_3):\ x_3\leq 0,\ r\leq 1 + |x_3|\Big \} \cup B\big ( (0,0,0),\, 1\big )
$$
such that
\begin{enumerate}[i)]
	\item $\vec\psi(r, x_3)=(r,x_3)$ on the half-line $r=1+|x_3|$, $x_3\leq 0$,
	\item $\vec\psi \big ( \sqrt{2}\sin (\bar\varphi), 1 + \sqrt{2}\cos (\bar\varphi) \big ) 
		= \Big ( \sin \big ( 2(\pi -\bar\varphi)\big ), \cos \big ( 2(\pi - \bar \varphi)\big ) \Big ) $,
		 $\frac{3\pi}{4}\leq \bar\varphi \leq \pi$,
	\item $\vec\psi\equiv \mathbf{id}$ in $\{\vec x:\ r^2 +(x_3-1)^2 \geq 8,\ x_3\leq 0,\ r\leq 1+|x_3|\}.$
\end{enumerate}

\smallskip
\noindent\textbf{Region} $e_\e:= \{ \vec e_3+ \rho \sin \varphi \vec e_r(\theta) + \rho \cos \varphi \vec e_3 
: \ \e < \rho< 1, \ 0\leq \varphi \leq \frac{\pi}{2}\}$.
Here we define $\vec u_\e$ in region $e_\e$ through its spherical coordinates
	\begin{align*}
&	u_\varphi^\e (\rho, \varphi) := \varphi,
	\qquad
	 u_\rho^\e (\rho ,\varphi):= \frac{1-\rho}{1-\e} u_\rho^\e(\e,\varphi) + \frac{\rho-\e}{1-\e} (2\cos\varphi + 6\e^{\gamma}),
	\\ &	u_\rho^\e(\e,\varphi)=
\Bigg ( 
		 ( \cos \varphi + 2\e^\gamma )^3 
		+ \left [ \frac{3r}{\partial_r \big ( -\cos f(r) \big ) }\right]_{r=g_\e(\varphi)}
 + 3\int_{s=0}^1 h_\e(s, \varphi)\dd s
	\Bigg )^{\frac{1}{3}},
	\end{align*}
	the function $h_\e(s,\varphi)$ being defined by
	
\begin{equation} \label{eq:def_h}
h_\e(s,\varphi):= \e \Big ( (1-s) \frac{g_\e(\varphi)}{\sin\varphi}  +s\e \Big ) \Big ( (1-s) g_\e'(\varphi)\cos\varphi
+ s (\e-g_\e(\varphi)\sin\varphi)\Big ),
\end{equation}
with $g_\e:[0,\frac{\pi}{2}]\to [0,\e]$ the inverse of the function $f_\e$ defined by
\begin{align}\label{eq:def_f_eps}
    f_\e (r):=\arctan \Big ( \frac{r}{\e^2} \Big ) + \alpha_\e \frac{r}{\e},
        \quad 0\leq r \leq \e,
        \qquad 
    \alpha_\e:=\arctan(\e).
\end{align}
	
\smallskip
\noindent\textbf{Region} $f$. 
 We define $\vec u_\e$ through its spherical coordinates
	\begin{align*}
	u_\rho^\e (\rho, \varphi) := 2\cos\varphi + \rho-1 + 6\e^\gamma,
	\qquad
	 u_\varphi^\e (\rho ,\varphi):= \varphi.
	\end{align*}
	Note that when $\rho=1$ the radial coordinate $$u_\rho^\e(1,\varphi)=2\cos\varphi + 6\e^{\gamma}$$
	coincides with
	the definition given in region $e_\e$.

\smallskip
\noindent\textbf{Region} $d_\e :=\{ x_1^2 +x_2^2 >\e^2,\ 0<x_3<1\}$.
We use cylindrical coordinates in the domain. The approximating map 
    \begin{equation}\label{eq:def-u_eps-region-d}
    \vec u_\e \big ( r\vec e_r(\theta) + x_3 \vec e_3\big ) =
        \vec w_\e \Big ( \vec g(\hat r\vec e_r(\theta) + x_3\,\vec e_3)\Big )
		=w_r^\e(s,z) \vec e_r(\theta) + w_3^\e(s,z)\, \vec e_3,
    \end{equation}
is defined by composing three different auxiliary transformations. The first one is given by 
	\begin{equation}
 \label{eq:rPrimeRegionD}
		\hat r=\begin{cases} \displaystyle \frac{(r-\e)r}{\e^{2\gamma}-\e}, & \e < r\leq \e^{2\gamma},\\ r, & r\geq \e^{2\gamma}.
		\end{cases}
	\end{equation}
After this transformation the radial distance lies in the same interval $(0,\infty)$ for all $\e$. 
The second transformation is the fixed axisymmetric map 
\begin{align*}
   \vec g (\hat r\vec e_r(\theta) + x_3\vec e_3) 
   =
   s(\hat r, x_3)\vec e_r 
   + z(\hat r, x_3)\vec e_3
\end{align*}
used before in region \(d\).
The third transformation is 
	\[
		\vec w_\e\big ( s\,\vec e_r(\theta) + z\,\vec e_3\big ) = w_r^\e(s,z)\,\vec e_r(\theta) + w_3^\e(s,z)\, \vec e_3
	\]
	\begin{equation}
    	\label{eq:defWr}
	 \begin{aligned}
		w_r^\e(s,z) &= \omega_\e(z) + s\,\sin \big ( \varphi(z)\big ),
		\\
		w_3^\e(s, z) &= -s\,\cos \big ( \varphi(z)\big ),
	 \end{aligned}
	\qquad \text{with}\quad
\varphi(z):= \tfrac{\pi}{4} \left (1+\tfrac{z}{3}\right ).	
	\end{equation}
The function 
$\omega_\e:[0,3]\to\R$ 
in \eqref{eq:defWr} is 
defined as
    \begin{align*}
            \omega_\e (z) = \begin{cases}
            z (2\e^\gamma) + (1-z)\e^\gamma, & 0\leq z\leq 1,
            \\
            \left( (2\e^\gamma)^3 + (z-1)\cdot \displaystyle \frac{3\e}{f_\e'(\e)} \right)^{1/3}, & 
            1\leq z\leq 2,
            \\
            (3-z)\eta_\e + (z-2)(6\e^{\gamma}),& 2\leq z\leq 3,
	   \end{cases}
    \end{align*}
   with \(f_\e\) defined in \eqref{eq:def_f_eps}. Notice that it is piecewise $C^1$ and 
    \begin{align*} 
    \omega_\e(0)=\e^\gamma,\quad \omega_\e(1)=2\e^\gamma,
        \quad \omega_\e(2)=\eta_\e
        =2\e^\gamma + o(\e^\gamma),
        \quad \omega_\e(3)=6\e^{\gamma}.
    \end{align*}
  Also, $\omega_\e(z)$ is increasing and, because of that,
  \begin{align}
        \label{eq:boundOmegaEps}
      \e^\gamma \leq \omega_\e(z)\leq 6\e^\gamma\quad \text{for all}\ z\in [0,3].
  \end{align}
In addition, as mentioned in \cite[Section~3.3.8]{BaHEMoRo_24PUB}, using that $\gamma\leq \frac{1}{3}$ it can be checked that 
  \begin{equation}
    \label{eq:omega-prime}
    |\omega'_\e(z)|\le c\,\e^\gamma \qquad\text{for all }z\in[0,3].
  \end{equation}

The definition of $\vec u_\e$  matches the one
already provided in regions $a_\e$, $b$, $c_\e$,  $e_\e$, and $f$.
In particular,
\begin{itemize}
\item On the interface between $d_\e$ and $b$, the deformation $\vec u_\e$ is:
    $$
    r\geq 1\quad \Rightarrow\quad \vec u_\e\big ( r\,\vec e_r(\theta) \big )
    = \e^\gamma\vec e_r(\theta)+ (r-1) \frac{\vec e_r(\theta) - \vec e_3}{\sqrt{2}}.
        $$
\item On the interface between $e_\e$ and $d_\e$, the deformation is: 
	$$
		\e\leq r\leq 1\ \Rightarrow\ 
		\vec u_\e\big (\vec e_3 + r\,\vec e_r(\theta)\big ) = \left(\frac{1-r}{1-\e} \eta_\e + \frac{r-\e}{1-\e} 6\e^{\gamma}\right)\vec e_r(\theta),
	$$
	with 
	\begin{equation}\begin{split}
        \label{eq:defEta}
\eta_\e  := 
    & 
u_\rho^\e(\e,\frac{\pi}{2})
    = \Bigg ( (2\e^\gamma)^3  + \frac{3r}{\sin f_\e(r)f'_\e(r)}\Bigg |_{r=\e}\Bigg )^{\frac{1}{3}}
      = \Bigg ( 
		(2\e^\gamma)^3
	 + \frac{3\e}{f'_\e(\e)}
	 \Bigg )^{\frac{1}{3}}
     \\ 
     =
     &
     2\e^\gamma + o(\e^\gamma).
     \end{split}
	\end{equation}
	This relevant quantity $\eta_\e$ is of order $\e^{\gamma}$ (this is why we chose  \(\gamma\leq 1/3\)).

	\item On the interface with $f$, the deformation is prescribed to be
	$$
	 r\geq 1
	 \quad \Rightarrow \quad
	 \vec u_\e\big (\vec e_3 + r \,\vec e_r(\theta)\big )
		= \big(r-1 + 6\e^{\gamma}\big )\,\vec e_r(\theta).
	$$
\end{itemize}

\bigskip
\noindent \textbf{Energy bounds.}
In \cite[Section 3.3.3]{BaHEMoRo_24PUB} it is shown that in those parts of the domain where $\det D\vec v$ is close to zero, a lower bound of the form 
$$
    \det D\vec u_\e(\vec x)\geq c\det D\vec v(\vec x)
$$
holds. Because of that, the energy
of $\vec u_\e$ is controlled.

\begin{remark}[The scale $\e^\gamma$]\label{rem:origin-eps-gamma}
In the construction we fixed an exponent $0<\gamma\le \frac13$ and we repaired the singular
behavior of $\vec v$ by \emph{opening} the collapsed set (the axis/neck) at a small geometric
scale of order $\e^\gamma$.

More precisely, the maps $\vec u_\e$ coincide with $\vec v$ away from the singular set, but
differ from $\vec v$ by shifts/extrusions whose amplitude is $\sim\e^\gamma$:
\begin{itemize}
\item in region $b$, the auxiliary map $\vec\phi_\e$ differs from $\vec v$ by the translations
$\sqrt2\,\e^\gamma$ along the direction $\vec d(\varphi,\theta)$ and $\e^\gamma$ in the $\vec e_3$--direction
(see \eqref{eq:def-u_eps-region-b});
\item in the neck $d_\e$, the ``horizontal extrusion'' function satisfies
$\e^\gamma\le \omega_\e(z)\le 6\e^\gamma$ for all $z\in[0,3]$ (see \eqref{eq:boundOmegaEps}),
so that the radius of the opened tube is comparable to $\e^\gamma$;
\item the matching radius at the interface between $e_\e$ and $d_\e$ is
$\eta_\e=2\e^\gamma+o(\e^\gamma)$ (see \eqref{eq:defEta}), which is the reason for the restriction
$\gamma\le \frac13$ in \cite{BaHEMoRo_24PUB}.
\end{itemize}
\end{remark}


\section{Patch of the dipole}\label{app:proof theo approx}

We provide here the detailed construction of the maps $\vec b_\delta$ provided in Theorem \ref{theo: approx} and the estimate of their energy.

\begin{proof}[Proof of Theorem \ref{theo: approx}]

Given the complexity of the proof, we first outline our overall strategy.
\begin{enumerate}[(i)]
\item \textbf{Starting point:} The Conti--De Lellis dipole $\vec v$ is a singular map 
that fails to be one-to-one only on the boundary of the ``bad'' set $U$.
Outside any neighborhood $U_\delta$ of $U$, the map $\vec v$ is a bi-Lipschitz homeomorphism onto its image.
By the construction in \cite{BaHEMoRo_24PUB} recalled in the Appendix \ref{app:dipole approx}, 
there exist regular approximations $\vec u_\e \in \mathcal{A}^r$ with uniformly bounded energy, i.e.,
$\sup_{\e>0} E(\vec u_\e) < \infty$.

\smallskip
\item \textbf{Definition of $\vec b_\delta$:} 
For each $\delta \in (0,1]$, we choose $\e > 0$ 
sufficiently small so that $\delta \geq c_0 \e^\gamma$ for a constant $c_0$ large enough. 
The scale $\e^\gamma$ is related to the energy concentration produced by $\vec v$; see Remark \ref{rem:origin-eps-gamma}.
Then we define
\begin{equation}\label{eq:bdelta-core-ext}
\vec b_\delta(\vec x) := 
\begin{cases}
\vec u_\e(\vec x) & \text{if } \dist(\vec x, U) \leq \delta/2, \\[4pt]
\text{interpolation} & \text{if } \delta/2 \leq \dist(\vec x, U) \leq \delta, \\[4pt]
\vec v(\vec x) & \text{if } \dist(\vec x, U) \geq \delta.
\end{cases}
\end{equation}

The key point is that the interpolation introduces gradient perturbations of order $\e^\gamma/\delta$, 
which remain controlled thanks to the scaling $\delta \geq c_0\e^\gamma$.

The transition layer $T_\delta := U_\delta \setminus U_{\delta/2}$ is handled separately in each 
of the regions $b$, $d$, $e$, and $f$, using interpolations adapted to the local 
geometry and coordinate systems.

\smallskip
\item \textbf{Verification:} We show that $\vec b_\delta$ is continuous across all the interfaces, 
bi-Lipschitz and orientation-preserving (by checking $\det D\vec b_\delta > 0$ in each region), 
and has energy bounded independently of $\delta$. 
\end{enumerate}

\smallskip
We now proceed with the detailed construction.

Fix $\delta\in(0,1]$ and let $U_\delta:=\{\vec x\in B(\vec 0,4)\colon \dist(\vec x,U)<\delta\}$ as in the statement (see Figure~\ref{fig:regions-delta}).  
Let $0<\gamma\le 1/3$ 
be the exponent used in the definition of the approximating maps $\vec u_\e$ in the Appendix \ref{app:dipole approx}.

\textbf{Step~1: definition of $\vec b_\delta$.}

In this step we define $\vec b_\delta$ in the transition layer
\[
  T_\delta:=U_\delta\setminus U_{\delta/2}=\{\vec x\in B(\vec0,4)\colon \delta/2\le\dist(\vec x,U)\le\delta\}.
\]

We choose $\e=\e(\delta)>0$ so that  
\begin{equation}\label{eq:delta-eps}
\delta \ge c_0\,\e^\gamma
\end{equation}
for a suitable constant $c_0$ that we tune along the construction.
In the rest of the proof we abbreviate $\vec u:=\vec u_\e$.
Let $\chi_\delta:[0,\infty)\to[0,1]$ be the scalar cut-off function defined by
\[
\chi_\delta(t):=
\begin{cases}
1, & 0\le t\le\delta/2,\\[1mm]
\displaystyle\frac{\delta-t}{\delta/2}, & \delta/2\le t\le\delta,\\[1mm]
0, & t\ge\delta .
\end{cases}
\]

Notice the scale that governs the choice of $\e=\e(\delta)$ in the patching:
in the transition layer $T_\delta$ the cut-off produces gradient terms of the form
$(\vec u_\e-\vec v)\otimes \nabla\chi_\delta$.
Since $|\nabla\chi_\delta|\sim \delta^{-1}$ and $|\vec u_\e-\vec v|\sim \e^\gamma$ on the relevant interfaces,
the perturbation is of size $\e^\gamma/\delta$.
The compatibility condition \eqref{eq:delta-eps} is therefore the
natural requirement ensuring that the interpolation does not create excessive strains and that the
Jacobian remains positive throughout $T_\delta$.

\smallskip
We define $\vec b_\delta$ separately in the four regions $b$, $d$, $e$, and $f$.
 
\smallskip
\noindent\textbf{Transition layer in region $b$.}

In spherical coordinates, region $b$ is given by
\[
\{\,\vec x=\rho\sin\varphi\,\vec e_r(\theta)+\rho\cos\varphi\,\vec e_3:
\ 1<\rho\le 3,\ \tfrac{\pi}{2}\le\varphi\le\pi\,\},
\]
and $\dist(\vec x,U)=\rho-1$. Set $s:=\rho-1$ and recall from Appendix \ref{app:dipole approx} 
the auxiliary maps used in the definition of $\vec u$ in region $b$:
\begin{equation*}\begin{split}
\vec d(\varphi,\theta)&=\sin\left(\frac{\varphi + \pi}{2} \right)\,\vec e_r(\theta)+\cos \left(\frac{\varphi + \pi}{2} \right)\,\vec e_3,\\[5pt]
\vec\phi_\e(\vec x)&=(s+\sqrt2\,\e^\gamma)\,\vec d(\varphi,\theta)+\e^\gamma\,\vec e_3.
\end{split}\end{equation*}
In this region $\vec v(\vec x)= s\,\vec d(\varphi,\theta)$. Moreover, $\vec v$ and $\vec u=\vec u_\e$ have angular component $\tfrac{\varphi+\pi}{2}$.

We define $\vec b_\delta$ on $b\cap T_\delta$ (i.e.,\ for $\delta/2\le s\le\delta$) by cutting off exactly
the two $\e^\gamma$--shifts that distinguish $\vec\phi_\e$ from $\vec v$:
\[
\vec b_\delta(\vec x)
:= s\,\vec d(\varphi,\theta)+\chi_\delta(s)\Big(\sqrt2\,\e^\gamma\,\vec d(\varphi,\theta)+\e^\gamma\,\vec e_3\Big)
= \big(s+\chi_\delta(s)\sqrt2\,\e^\gamma\big)\vec d(\varphi,\theta)+\chi_\delta(s)\e^\gamma\,\vec e_3.
\]

Then $\vec b_\delta=\vec v$ on $\{s=\delta\}$ (since $\chi_\delta(\delta)=0$). Also, on $\{s=\delta/2\}$
we have $\chi_\delta(\delta/2)=1$ and hence $\vec b_\delta=\vec\phi_\e$. Moreover,
\[
(\e^{-\gamma}\phi_r^\e)^2+(\e^{-\gamma}\phi_3^\e-1)^2
=\Big(\frac{s}{\e^\gamma}+\sqrt2\Big)^2,
\]
so under the choice $\delta\ge c_0\e^\gamma$ (with $c_0$ large enough) this quantity is $\ge 8$
when $s=\delta/2$, and therefore $\vec\psi\equiv\mathbf{id}$ on $(\e^{-\gamma}\vec\phi_\e)^{-1}(\{s=\delta/2\})$.
Consequently, $\vec u(\vec x)=\e^\gamma\vec\psi(\e^{-\gamma}\vec\phi_\e(\vec x))=\vec\phi_\e(\vec x)=\vec b_\delta(\vec x)$
on $\{s=\delta/2\}$, so $\vec b_\delta$ matches continuously with $\vec u$ across the interface.

\smallskip
\noindent\textbf{Transition layer in region $d$.}

Within the neck $d$, the limit map $\vec v$ is singular on the axis: it collapses the vertical segment $S$ to the origin, 
whereas $\vec u_\e$ ``opens'' the axis to form a cylindrical cavity of radius $\sim \e^\gamma$. 
A generic linear interpolation would introduce a gradient 
perturbation of order $\sim \e^\gamma/\delta$. 
However, we have also to control the determinant.
The explicit construction is thus necessary.

Let $\vec g$ be the auxiliary map defined in \eqref{eq:auxiliary g} and \eqref{eq:def_g_specific}.
By the bi-Lipschitz property of the planar representative of $\vec g$, after possibly decreasing $\e$ we can also ensure
the separation of the image through $\vec u_\e$ of the level set $\dist(\vec x, U) = \delta/2$ and the image through $\vec v$
of the level set 
\hbox{$\dist(\vec x, U) = \delta$}.
In particular, for \(\e\) small enough, $\delta/2\ge2\sqrt2\,\e^\gamma$, and for every $\vec x\in d\cap T_\delta$ we have $r=\dist(\vec x,U)\ge\delta/2$.
We set
\begin{equation*}
  \vec b_\delta(\vec x)=w_r(\vec x)\,\vec e_r(\theta)+w_3(\vec x)\,\vec e_3,
\end{equation*}
where 
\begin{equation}\label{eq:wr-def}
\begin{cases}
w_r(\vec x) :=\chi_\delta(s)\omega_\e(z)+s\sin\varphi(z),\\[3pt]
w_3(\vec x) :=-s(r,x_3)\cos\varphi(z(r,x_3)),
\end{cases}
\end{equation}
with $z=z(r,x_3)$ the vertical coordinate in the intermediate configuration of Figure \ref{fig:function g}.
The cut-off $\chi_\delta(s)=\frac{\delta-s}{\delta/2}$ reduces the ``horizontal extrusion''
$\omega_\e(z)\vec e_r$ from full magnitude at the level curve
$\dist(\vec x,U)=\tfrac{\delta}{2}$ (where $\chi_\delta=1$) to zero at the level curve
$\dist(\vec x,U)=\delta$ (where $\chi_\delta=0$). By the definition of $\vec v$ in
\eqref{eq:def-v-region-d} and the definition of $\vec u$ in
\eqref{eq:def-u_eps-region-d}, one has that $\vec b_\delta$ satisfies
\eqref{eq:bdelta-core-ext}.

\begin{figure}[ht]
\centering

\begin{tikzpicture}[
    scale=3,
    >={Stealth[length=2mm]},
    flow/.style={postaction={decorate, decoration={markings, mark=at position 0.55 with {\arrow{>}}}}},
    bflow/.style={postaction={decorate, decoration={markings, mark=at position 0.6 with {\arrow{>}}}}},
    curveflow/.style={
        decoration={
            markings,
            mark=at position 0.15 with {\arrow{Stealth[length=2mm]}},
            mark=at position 0.35 with {\arrow{Stealth[length=2mm]}},
            mark=at position 0.55 with {\arrow{Stealth[length=2mm]}},
            mark=at position 0.75 with {\arrow{Stealth[length=2mm]}},
            mark=at position 0.95 with {\arrow{Stealth[length=2mm]}}
        },
        postaction={decorate}
    },
    font=\small
]

\def\del{3}
\def\pointradius{0.5pt}
\def\varphiValue{45}
\def\epsgamma{\del/15}

\coordinate (O) at (0,0);
\coordinate (P) at (\del/2,0);
\coordinate (Q) at ($(P) + (6*\epsgamma,0) $);
\coordinate (R) at (\del, 0);
\coordinate (P') at ($({\del/2*sin(\varphiValue)}, {-\del/2*cos(\varphiValue)})$);
\coordinate (Q') at ($(P') + (\epsgamma,0) $);
\coordinate (R') at ($({\del*sin(\varphiValue)}, {-\del*cos(\varphiValue)})$);
\coordinate (R'') at ($({1.1*\del*sin(\varphiValue)}, {-1.1*\del*cos(\varphiValue)})$);

\fill (O) circle (\pointradius) node[below] {$O$};
\fill (Q) circle (\pointradius);
\fill (R) circle (\pointradius);
\fill (Q') circle (\pointradius);
\fill (R') circle (\pointradius);

\draw[->] (O) -- (1.1*\del,0) node[right] {$\varphi=\tfrac{\pi}{2}$};
\draw[->] (O) -- (R'') node[below right] {$\varphi=\tfrac{\pi}{4}$};

\draw[thick] (P) arc (0:-45:\del/2);
\draw[ultra thick, curveflow] (R') arc (-45:0:\del);

\draw[thick,bflow,red] (P') -- ($(P') + (\epsgamma,0) $);
\draw[thick,bflow,red] (P) -- (Q);

\foreach \ang/\t in {45/0.0, 54/0.2, 63/0.4, 72/0.6, 81/0.8, 90/1.0}{
  \pgfmathsetmacro{\len}{\epsgamma + \t*(6*\epsgamma-\epsgamma)}
  \coordinate (A\ang) at ({\del/2*sin(\ang)},{-\del/2*cos(\ang)});
  \coordinate (B\ang) at ($(A\ang)+(\len,0)$);
  \fill (B\ang) circle (\pointradius);
  \draw[thick,bflow,red] (A\ang) -- (B\ang);

  \coordinate (R\ang) at ({\del*sin(\ang)},{-\del*cos(\ang)});
  \fill (R\ang) circle (\pointradius);

  \draw[dashed,gray] (B\ang) -- (R\ang);
}

\draw[
    ultra thick,
    blue,
    curveflow
]
plot[smooth] coordinates {
    (B45)
    (B54)
    (B63)
    (B72)
    (B81)
    (B90)
};

\coordinate (A6372) at ($(A63)!0.6!(A72)$);
\node[left=2pt of A6372] {${\bm v}(\{s=\tfrac{\delta}{2}\})$};
\coordinate (B6372) at ($(B63)!0.2!(B72)$);
\node[right=2pt of B6372] {${\bm u}(\{s=\tfrac{\delta}{2}\})$};
\coordinate (R6372) at ($(R63)!0.5!(R72)$);
\node[right=1pt of R6372] {${\bm v}(\{s=\delta\})$};
\coordinate (OR') at ($(O)!0.6!(R')$);
\node[below left=0pt of OR'] {${\bm v}(\{x_3=0\})$};

\draw[decorate, decoration={brace, amplitude=5pt}] (P) -- (Q)
    node[midway, yshift=5pt, above,red] {$6\e^\gamma$};

\end{tikzpicture}

\caption{The red arrows illustrate the effect of the function $\omega_\e(z)$,
which extrudes horizontally the image by $\vec v$ of the level curve $\{\vec x\in\Om: s=\dist(\vec x, U)=\tfrac{\delta}{2}\}$.
This gives rise to the blue curve, which is the image of the same level set by the approximating map $\vec u$.
Our construction $\vec b_\delta$ interpolates linearly between that blue curve and 
the curve $\vec v(\{s=\delta\})$ on the right-end of the transition layer $T_\delta$.
}
\label{fig:extrusion}
\end{figure}
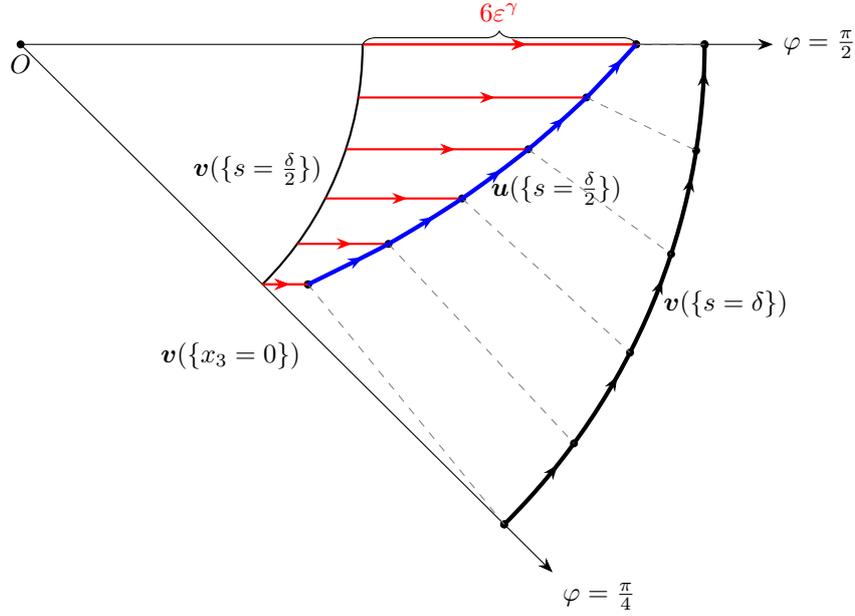

The idea is the following. In the image by $\vec v$ of the level curve at distance $s=\tfrac{\delta}{2}$,
the zenithal angle $\varphi$ (angle with respect to the downwards vertical semiaxis) goes from the ray $\varphi=\frac{\pi}{4}$ to the ray $\varphi=\frac{\pi}{2}$ (see Figure \ref{fig:regions}) for region $d$ after the dipole deformation). 
This is reflected in the formula $\varphi(z)=\tfrac{\pi}{4}(1+\frac{z}{3})$. The bottom-right point of the level curve $s=\tfrac{\delta}{2}$ 
in the reference configuration, namely, $(r=1+\tfrac{\delta}{2}, x_3=0)$, will be mapped to $z=0$, ray $\varphi=\frac{\pi}{4}$, and 
radial coordinate $s=\tfrac{\delta}{2}$. Similarly, the top-right point of the level curve $s=\tfrac{\delta}{2}$ will be mapped to $z=3$, 
ray $\varphi=\frac{\pi}{2}$, and radial coordinate $s=\tfrac{\delta}{2}$. 
These two points need to be joined by a circular arc of parameters $s=\tfrac{\delta}{2}$ and $\tfrac{\pi}{4}\leq \varphi\leq \tfrac{\pi}{2}$.

The image by $\vec v$ of the level curve at distance $s=\delta$ is also a circular arc, of parameters $s=\delta$ 
and $\tfrac{\pi}{4}\leq \varphi\leq \tfrac{\pi}{2}$.
In our construction, on the right circular arc $\vec b_\delta$ must coincide with $\vec v$.
On the left circular arc instead $\vec b_\delta$ is not the image curve $\tfrac{\delta}{2}$, since it must coincide with $\vec u$. 
Here $\vec u$ takes 
that curve and ``extrudes it horizontally'' by a distance $\omega_\e(z)$ that begins being $\e^\gamma$ when $x_3=0$, $z=0$, 
and ends up being $6\e^\gamma$ when $x_3=1$, $z=3$ 
(see Figure \ref{fig:extrusion}).
The idea, then, is to gradually reduce the magnitude of the horizontal 
extrusion, from $\omega_\e \sim \e^\gamma$ at distance $\tfrac{\delta}{2}$, to zero at distance $\delta$.
An observation to be made is that the branch $\hat r$  in the definition \eqref{eq:rPrimeRegionD} 
of $\vec u_\e$ in region $d$ does not affect us in the transition layer $\tfrac{\delta}{2}\leq s\leq \delta$. 
Indeed, $\hat r=r$ in the branch $r\geq \e^{2\gamma}$, and because of the assumption \eqref{eq:delta-eps}, in the region \(s(r,x_3)\geq \delta/2\) we have \(r\geq \delta/2\geq \e^{2\gamma}\), for \(\e\) small enough.

\smallskip
\noindent\textbf{Transition layer in region $e$.}

Taking $\e$ small enough, we can assume that $T_\delta$ does not intersect the region $e'_\e$.
In $e_\e$, we have $\dist(\vec x,U)=\rho\cos\varphi$, where \( (\rho, \theta, \varphi)\) are the spherical coordinates in this region.
We set
\begin{equation*}
\vec b_\delta(\vec x)
 = \bigl[ (1-\chi_\delta(\rho\cos\varphi))\,v_\rho(\rho,\varphi) + \chi_\delta(\rho\cos\varphi)\,u_\rho(\rho,\varphi) \bigr]
   \bigl(\sin\varphi\,\vec e_r(\theta)+\cos\varphi\,\vec e_3\bigr).
\end{equation*}
This matches $\vec u_\e$ when $\dist(\vec x,U) \le \delta/2$ and $\vec v$ when $\dist(\vec x,U) \ge \delta$.

\smallskip
\noindent\textbf{Transition layer in region $f$.}

In this region the transition is straightforward, since $\vec v$ is regular and
$\vec u_\e\to \vec v$ in $C^1$.
Let
\[
\chi(\vec x):=\frac{\delta-\dist(\vec x,U)}{\delta/2}.
\]
For $\vec x\in T_\delta\cap f$, we set
\begin{equation*}
  \vec b_\delta(\vec x):=(1-\chi(\vec x))\,\vec v(\vec x)+\chi(\vec x)\,\vec u(\vec x).
\end{equation*}
Using the explicit formulas in the Appendices \ref{app:dipole full} and
\ref{app:dipole approx}, one can check that
$|\vec u-\vec v|\le c\,\e^\gamma$ and $|D\vec u-D\vec v|\le c\,\e^\gamma$.

\medskip
\textbf{Step~2: uniform bound on the energy and orientation-preserving character}.

We have to show that the energy contribution coming from the transition layer is bounded independently of $\delta$. 

\smallskip
\noindent\textbf{Energy region $b$.}

In $b\cap T_\delta$ we write
\[
\vec x=\rho\sin\varphi\,\vec e_r(\theta)+\rho\cos\varphi\,\vec e_3,
\qquad \frac{\pi}{2}\le \varphi\le \pi,
\qquad s:=\rho-1=\dist(\vec x,U)\in\Big[\frac{\delta}{2},\delta\Big].
\]
Recall that in $b$ one has
\[
\vec b_\delta(\vec x)
=\Big(s+\chi_\delta(s)\sqrt2\,\e^\gamma\Big)\,\vec d(\varphi,\theta)
+\chi_\delta(s)\,\e^\gamma\,\vec e_3,
\qquad 
\vec d(\varphi,\theta):=\sin\Phi(\varphi)\,\vec e_r(\theta)+\cos\Phi(\varphi)\,\vec e_3,
\]
with $\Phi(\varphi):=\frac{\varphi+\pi}{2}$.
Set
\[
r_\delta(s):=s+\chi_\delta(s)\sqrt2\,\e^\gamma,
\qquad z_\delta(s):=\chi_\delta(s)\e^\gamma,
\qquad \rho=1+s.
\]
Since $\chi_\delta'(s)=-2/\delta$ on $[\delta/2,\delta]$ and $\delta\ge c_0\e^\gamma$ (see \eqref{eq:delta-eps}),
we have
\[
|r_\delta'(s)|\le 1+c\,\frac{\e^\gamma}{\delta}\le c,
\qquad |z_\delta'(s)|\le c\,\frac{\e^\gamma}{\delta}\le c.
\]
Moreover,
\[
\p_s\vec b_\delta=r_\delta'(s)\,\vec d+z_\delta'(s)\,\vec e_3,
\qquad
\p_\varphi\vec b_\delta=\frac12\,r_\delta(s)\big(\cos\Phi\,\vec e_r(\theta)-\sin\Phi\,\vec e_3\big),
\qquad
\p_\theta\vec b_\delta=r_\delta(s)\sin\Phi\,\p_\theta\vec e_r(\theta).
\]
Using the metric factors of spherical coordinates and $\rho=1+s\simeq 1$ on $T_\delta$,
and using that, for \(\frac{\pi}{4}\leq \varphi \leq \pi\),
\begin{equation} \label{eq:bound_sinus} 
\frac{\sin^2\Phi}{\sin^2 \varphi} =\frac{\cos^2\left( \frac{\varphi}{2}\right)}{\sin^2\varphi} 
= \frac{1}{4\sin^2\left( \frac{\varphi}{2}\right)}\leq 2, 
\end{equation}
we find that 
\[
|D\vec b_\delta(\vec x)|^2
\le c\,\Big(|\p_s\vec b_\delta|^2+\frac{1}{\rho^2}|\p_\varphi\vec b_\delta|^2+\frac{1}{\rho^2\sin^2\varphi}|\p_\theta\vec b_\delta|^2\Big)
\le c\,\Big(1+\frac{\e^{2\gamma}}{\delta^2}\Big),
\]
and therefore
\[
\int_{T_\delta\cap b}|D\vec b_\delta|^2\,\dd\vec x
\le c\,\Big(\delta+\frac{\e^{2\gamma}}{\delta}\Big)\le c.
\]

For the determinant, in the coordinates $(s,\varphi,\theta)$ (see, e.g., \cite[Appendix B.3]{BaHEMoRo_24PUB}) one has
\begin{align}
\det D\vec b_\delta(\vec x)
&=\frac{\det(\p_s\vec b_\delta,\p_\varphi\vec b_\delta,\p_\theta\vec b_\delta)}{\rho^2\sin\varphi}\notag\\
&=\frac12\,\frac{r_\delta(s)^2}{\rho^2}\,\frac{\sin\Phi(\varphi)}{\sin\varphi}
\Big(r_\delta'(s)+z_\delta'(s)\cos\Phi(\varphi)\Big)\notag\\
&=\frac12\,\frac{r_\delta(s)^2}{\rho^2}\,\frac{\sin\Phi(\varphi)}{\sin\varphi}
\Big(1+\e^\gamma\chi_\delta'(s)\big(\sqrt2+\cos\Phi(\varphi)\big)\Big).\label{eq:det-b-layer-b}
\end{align}
Since $\frac{\sin\Phi(\varphi)}{\sin\varphi}=\frac{1}{2\sin \left( \frac{\varphi}{2}\right)}$ (see \eqref{eq:bound_sinus}) is bounded above and below on $[\pi/2,\pi]$,
choosing $c_0$ in \eqref{eq:delta-eps} large enough yields
\[
1+\e^\gamma\chi_\delta'(s)\big(\sqrt2+\cos\Phi(\varphi)\big)\ge \frac12
\qquad\text{for }s\in[\delta/2,\delta].
\]
We infer from \eqref{eq:det-b-layer-b} that $\det D\vec b_\delta(\vec x)\ge c\,s^2$ on $T_\delta\cap b$.
Hence, 
\[
\int_{T_\delta\cap b}H(\det D\vec b_\delta)\,\dd\vec x
\le c\,\int_{\delta/2}^{\delta}s^{-2\alpha}\,\dd s \leq c,
\]
independently of~$\delta$.

\smallskip
\noindent\textbf{Energy region $d$.}

Recall that in $d\cap T_\delta$ by \eqref{eq:delta-eps} we have $r\geq \frac{\delta}{2}\geq 2\sqrt{2}\e^\gamma$.
By the formula for the gradient in 
cylindrical-cylindrical coordinates (see, e.g., \cite[Appendix B.2]{BaHEMoRo_24PUB}),
\[
    D\vec b_\delta (\vec x)
    = 
    \begin{pmatrix}
        \partial_r w_r & 0 & \partial_{x_3} w_r
        \\
        0 & \tfrac{w_r}{r} & 0
        \\
        \partial_r w_3 & 0 & \partial_{x_3} w_3
    \end{pmatrix} .
\]
In particular, 
\begin{equation}
    \label{eq:detstretch}
\det D\vec b_\delta = \frac{w_r}{r} \frac{\partial(w_r,w_3)}{\partial(r,x_3)}.
\end{equation}
\medskip

Let us first study the tangential stretch $\frac{w_r}{r}$. Thanks to \eqref{eq:boundOmegaEps}, we have
\begin{gather*}
	 w_r(\vec x) = \frac{\delta - s(r,x_3)}{\delta/2}
		\omega_\e\big ( z(r,x_3)\big ) + s(r,x_3)\sin\varphi(z(r,x_3))
	\\ \Rightarrow \quad
		s\sin \varphi \leq w_r(\vec x) \leq \frac{\delta/2}{\delta/2}\cdot 6\e^\gamma +s\sin\varphi
\end{gather*}
\begin{align}
    \label{eq:estimateStretch}
	\Rightarrow \quad
		\frac{s\sin \varphi}{r} \leq \frac{w_r(\vec x)}{r} \leq 2 +\frac{s}{r}\sin\varphi.
\end{align}
The lower bound for $\tfrac{w_r(\vec x)}{r}$ will be used to compare $\det D\vec b_\delta$ with $\det D\vec v$. 
By construction of $\vec b_\delta$, when we are close to the axis, we have $s=r$ (see \eqref{eq:def_g_specific}). In the rest of the region,
$s$ is still controlled by $r$; see Appendix \ref{app:dipole full}. Therefore, $\tfrac{w_r(\vec x)}{r}$ is bounded from above by a constant, independently of $\e$.

Regarding the $2 \times 2$ minor in the expression for $\det D\vec b_\delta$, observe that    
\[
    \frac{\partial(w_r,w_3)}{\partial(r,x_3)}
    = \frac{\partial(w_r,w_3)}{\partial(s,z)} \frac{\partial(s,z)}{\partial(r,x_3)} .
\]
Note also that 
\[
    \frac{\partial(s,z)}{\partial(r,x_3)} = \frac{\partial \vec g}{\partial r}\wedge \frac{\partial \vec g}{ \partial x_3}
\]
since $(r,x_3)\mapsto (s,z)$ is the 2D representation of the auxiliary transformation 
$\vec g(r\vec e_r+x_3\vec e_3) = s(r,x_3)\vec e_r + z(r,x_3)\vec e_3$. This 2D representation is bi-Lipschitz, as shown in Figure \ref{fig:function g}. Hence,
\[
    c^{-1}\frac{\partial(w_r,w_3)}{\partial(s,z)} \leq \frac{\partial(w_r,w_3)}{\partial(r,x_3)} \leq c\,\frac{\partial(w_r,w_3)}{\partial(s,z)}.
\]
On the other hand,
\begin{align}
    \frac{\partial(w_r,w_3)}{\partial(s,z)}
    &= 
    \begin{vmatrix}
        -\tfrac{2}{\delta}\omega_\e(z) + \sin\varphi 
        &
        2\frac{\delta-s}{\delta}\omega_\e'(z) + \tfrac{\pi}{12}s\cos\varphi
        \\
        -\cos \varphi & \frac{\pi}{12}s\sin\varphi
    \end{vmatrix}\nonumber
    \\ 
    &= 
        \label{eq:explicitDetOmegaPrime}
    \tfrac{\pi}{12} s 
        -\tfrac{2}{\delta} \omega_\e(z)\cdot \tfrac{\pi}{12}s\sin\varphi
        + 
        \underbrace{\cos\varphi\cdot 2\tfrac{\delta-s}{\delta}\omega_\e'(z)}_{\ge 0}
    \\
        \label{eq:LBabsorb}
    & \geq 
    \tfrac{\pi}{12}s \Bigg ( 1- 2\frac{\omega_\e(z)}{\delta}\sin\varphi \Bigg)
    \underset{\text{\eqref{eq:boundOmegaEps}}}
    {\geq }
    \tfrac{\pi}{12}s \Bigg ( 1- 2\frac{6\e^\gamma}{\delta}\sin\varphi \Bigg)
    \underset{\text{\eqref{eq:delta-eps}}}{\geq}
    c\,s.
\end{align}
Since $\tfrac{\pi}{4}\leq \varphi \leq \tfrac{\pi}{2}$,
combining \eqref{eq:detstretch}, 
\eqref{eq:estimateStretch}, and \eqref{eq:LBabsorb}
yields that
\[
    \det D\vec b_\delta (\vec x)\geq c\,\frac{s^2(r,x_3)\sin\varphi(z(r,x_3))}{r}
    \geq c\,\det D\vec v(\vec x)
    \geq c\,\frac{s^2(r,x_3)}{r}.
\]
Similarly, from 
\eqref{eq:omega-prime},
\eqref{eq:detstretch}, 
\eqref{eq:estimateStretch}, 
and
\eqref{eq:explicitDetOmegaPrime},
we obtain the upper bound
\[
    \det D\vec b_\delta(\vec x) \leq c.
\]
Therefore, $\int_{T_\delta\cap d} H(\det D\vec b_\delta)\leq c$ independently of~$\delta$
(compare with \eqref{eq:integrability_H}).    
After a change of coordinates from $(r,x_3)$ to $(s,z)$,
using \eqref{eq:boundOmegaEps}
and \eqref{eq:omega-prime},
each term on $D\vec b_\delta$ other than the tangential stretch $\tfrac{w_r}{r}$
can be controlled.

\smallskip
\noindent\textbf{Energy region $e$.}

In this region the map $\vec v$ degenerates at the boundary $\varphi=\pi/2$.
We use spherical coordinates. The radial component is given by
$b_\rho = \chi_\delta u_\rho + (1-\chi_\delta) v_\rho$.
Since the angular components coincide ($b_\varphi = v_\varphi = u_\varphi = \varphi$), the Jacobian\footnote{See, e.g., \cite[Appendix B.3]{BaHEMoRo_24PUB}} is
\begin{equation*}
    \det D\vec b_\delta = \frac{(b_\rho)^2}{\rho^2} \p_\rho b_\rho.
\end{equation*}
We must check that $\p_\rho b_\rho > 0$. Differentiating the interpolation gives
\begin{equation*}
    \p_\rho b_\rho = \chi_\delta \p_\rho u_\rho + (1-\chi_\delta) \p_\rho v_\rho + \chi'_\delta (u_\rho - v_\rho) \p_\rho \dist(\vec x, U).
\end{equation*}
In this region, $\dist(\vec x, U) = \rho \cos\varphi$, so that $\p_\rho \dist(\vec x, U) = \cos\varphi$.
Specifically, we have the following estimates:
\begin{enumerate}[(i)]
    \item For the convex combination term: since $\p_\rho u_\rho \ge \cos\varphi$
    (see \cite[p.~39, Eq.~(3.33)]{BaHEMoRo_24PUB}) and $\p_\rho v_\rho = \cos\varphi$, we have
    \[
    \chi_\delta \p_\rho u_\rho + (1-\chi_\delta) \p_\rho v_\rho \ge \cos\varphi.
    \]
    \item For the perturbation term: we have $\chi'_\delta = -2/\delta$ and 
    $|u_\rho - v_\rho| \le 7\e^\gamma$. 
    Indeed,
    writing 
    \begin{align*}
u_\rho(\rho,\varphi)-v_\rho(\rho,\varphi)
    = \frac{1-\rho}{1-\e}(u_\rho(\e,\varphi) - \cos\varphi )
	+ \frac{\rho-\e}{1-\e}\cdot 6\e^\gamma
	 -\frac{\e(1-\rho)}{1-\e}\cos\varphi
    \end{align*}
    and considering that 
    $\cos\varphi\leq u_\rho(\e, \varphi)\leq \cos\varphi + 6\e^\gamma$ (see \cite[Eq.~(3.10)]{BaHEMoRo_24PUB}),
    it is possible to see that 
    \begin{align*}
u_\rho(\rho,\varphi)-v_\rho(\rho,\varphi)
	\geq  \frac{-\e(1-\rho)}{1-\e}\cos\varphi
	\geq -2\e\cos\varphi.
\end{align*}
    On the other hand,
    \begin{align*}
 u_\rho - v_\rho
	\leq
	\frac{1-\rho}{1-\e}\cdot 6\e^\gamma
	+ \frac{\rho-\e}{1-\e}\cdot 6\e^\gamma
	+ \frac{1}{1-\e}\cdot \e^\gamma\cdot 1
	\leq 7\e^\gamma.
\end{align*}
    Consequently,
    \[
    \chi'_\delta (u_\rho - v_\rho) \p_\rho \dist(\vec x, U) \ge - \frac{2}{\delta} (7\e^\gamma) \cos\varphi = -\frac{14\e^\gamma}{\delta}\cos\varphi.
    \]
\end{enumerate}
Combining these, we factor out the term $\cos\varphi$:
\begin{equation*}
    \p_\rho b_\rho \ge \cos\varphi \left( 1 - \frac{14\e^\gamma}{\delta} \right).
\end{equation*}
Since $\cos\varphi>0$ for $\varphi\in[0,\pi/2)$, we have that $\p_\rho b_\rho$
is comparable to $\p_\rho v_\rho$ up to a strictly positive multiplicative constant,
provided we choose $c_0>14$ in \eqref{eq:delta-eps}.
In addition, the lower bound on $u_\rho-v_\rho$ established above gives
\begin{align*}
b_\rho
=\chi_\delta u_\rho+(1-\chi_\delta)v_\rho
= v_\rho + \chi_\delta(u_\rho - v_\rho)
\ge v_\rho - 2\e\cos\varphi
= (1+\rho-2\e)\cos\varphi.
\end{align*}
Since $\rho\le 1$ and $\e$ is small (in particular $\e<1/4$), we have
$1+\rho-2\e\ge \frac{1}{2}(1+\rho)$, and therefore
$b_\rho^2\ge \frac{1}{4}(1+\rho)^2\cos^2\varphi = \frac{1}{4}\,v_\rho^2$.
Hence
\[
\det D\vec b_\delta
=\frac{b_\rho^2}{\rho^2}\,\p_\rho b_\rho
\ge \frac{1}{4}\,\frac{v_\rho^2}{\rho^2}\cdot c\,\p_\rho v_\rho
= c\,\det D\vec v
\qquad\text{on }T_\delta\cap e.
\]

We now bound the $H$ term.
Since $\det D\vec b_\delta\ge c\,\det D\vec v$ on $T_\delta\cap e$ and $H$ is 
decreasing near~$0$, we have
$H(\det D\vec b_\delta)\le H(c\,\det D\vec v)$.
Therefore, from the integrability of $H(\det D\vec v)$ established 
in Appendix~\ref{app:dipole full}, $\int_{T_\delta\cap e}H(\det D\vec b_\delta)\,\dd\vec x\leq c$
independently of~$\delta$.

\smallskip
\noindent\textbf{Energy region $f$.}

In $f \cap T_\delta$ we use the convex interpolation
\[
  \vec b_\delta = (1-\chi)\vec v + \chi \vec u,
  \quad
  \chi(\vec x) := \frac{\delta - \dist(\vec x,U)}{\delta/2} \in [0,1],
  \quad
  |\nabla \chi| = \frac{2}{\delta}.
\]

Since
\[
  D\vec b_\delta
  = D\vec v + \chi\, D(\vec u - \vec v)
    + (\vec u - \vec v) \otimes \nabla\chi,
\]
and, by the explicit formulas in Appendices~\ref{app:dipole full} and~\ref{app:dipole approx},
$|\vec u - \vec v| \leq c\,\e^\gamma$ and
$|D\vec u - D\vec v| \leq c\,\e^\gamma$ in this region, we get
\[
  |D\vec b_\delta|
  \leq |D\vec v|
       + 6\e^\gamma\!\Big(\!\sqrt{2} + \frac{2}{\delta}\Big).
\]
Since the dipole $\vec v$ is $C^1$ in region~$f$, it follows that
\[
  \int_{T_\delta \cap f} |D\vec b_\delta|^2 \, d\vec x \leq c.
\]

For the determinant, note that
\[
  \vec b_\delta
  = b_\rho\!\big(\vec x(\rho,\varphi)\big)\,
    (\sin\varphi\, \vec e_r + \cos\varphi\, \vec e_3),
  \qquad
  b_\rho = v_\rho(\rho,\varphi) + \chi\!\big(\vec x(\rho,\varphi)\big)\cdot 6\e^\gamma.
\]
Hence $b_\rho \geq v_\rho(\rho,\varphi)$, and
\[
  \partial_\rho b_\rho
  = \partial_\rho v_\rho + 6\e^\gamma\, \nabla\chi \cdot \partial_\rho \vec x
  \geq \partial_\rho v_\rho - 6\e^\gamma\, |\nabla\chi|
  = 1 - \frac{12\e^\gamma}{\delta}.
\]
Therefore, if $c_0 > 12$ in \eqref{eq:delta-eps},
\begin{equation}\label{eq:detf}
  \det D\vec b_\delta
  = \frac{b_\rho^2}{\rho^2}\, \partial_\rho b_\rho
  \geq \frac{v_\rho^2}{\rho^2}
       \Big(1 - \frac{12\e^\gamma}{\delta}\Big)
  \geq c\, \det D\vec v
  \quad \text{on } T_\delta \cap f.
\end{equation}

For the $H$ term, since $H$ is decreasing near zero,
\eqref{eq:detf} gives
$H(\det D\vec b_\delta) \leq H(c\,\det D\vec v)$
on~$T_\delta \cap f$.
In particular, $\int_{T_\delta \cap f} H(\det D\vec b_\delta)\, d\vec x \leq c$
independently of~$\delta$.

\bigskip
Collecting the constraints on the constant $c_0$ in \eqref{eq:delta-eps}, we need:
\begin{itemize}
\item in region $b$, positivity of the Jacobian: $c_0>4\sqrt2$;
\item in region $d$, positivity of the Jacobian: $c_0>12$;
\item in region $e$, positivity of $\p_\rho b_\rho$: $c_0>14$;
\item in region $f$, positivity of $\p_\rho b_\rho$: $c_0>12$;
\item in region $b$, injectivity of the planar profile: $c_0>2\sqrt2$;
\item matching with $\vec\psi\equiv\mathbf{id}$ at $s=\delta/2$: $c_0\ge 4$.
\end{itemize}
Thus the binding condition is $c_0>14$.

\medskip
\textbf{Step~3: bi-Lipschitz properties.}

We prove that the map $\vec b_\delta$ constructed in Step~1 is bi-Lipschitz on $B(\vec 0,4)$.
Recall that $\vec b_\delta=\vec u$ on $U_{\delta/2}$ and $\vec b_\delta=\vec v$ on $B(\vec 0,4)\setminus U_\delta$.
Moreover, $\vec b_\delta$ is continuous across all interfaces by construction.

\smallskip
\noindent\emph{(i) $\vec b_\delta$ is Lipschitz.}
The domain $B(\vec 0,4)$ is partitioned into the following subdomains: the core $U_{\delta/2}$, the exterior
$B(\vec 0,4)\setminus U_\delta$, and the four pieces $T_\delta\cap b$, $T_\delta\cap d$, $T_\delta\cap e$, $T_\delta\cap f$.
On each of these pieces, $\vec b_\delta$ is 
piecewise
$C^1$ and the estimates of Step~2 (together with the explicit formulas
in Step~1) give the bound
\begin{equation}\label{eq:Dbdelta-Linfty}
\|D\vec b_\delta\|_{L^\infty(B(\vec 0,4))}\le C(\delta)<\infty.
\end{equation}
Since the traces match on the interfaces, no jump part appears in the distributional gradient; by \eqref{eq:Dbdelta-Linfty}, $\vec b_\delta$ is Lipschitz.

\smallskip
\noindent\emph{(ii) Pointwise bound on the inverse matrix.}
We apply the elementary estimate
\begin{equation}\label{eq:inv-bound}
|\vec A^{-1}|\le \frac{c\,|\vec A|^2}{\det \vec A} , \qquad \text{for any  } \vec A\in\R^{3\times 3} \text{ with } \det \vec A>0 
\end{equation}
to $\vec A=D\vec b_\delta(\vec x)$.
By construction, $\vec b_\delta$ coincides with $\vec u=\vec u_\e$ on $U_{\d/2}$
and with $\vec v$ on $B(\vec 0,4)\setminus U_\d$.
Since $\vec u$ is bi-Lipschitz on $B(\vec 0,4)$
(Appendix~\ref{app:dipole approx}, property~(i)) and $\vec v$ is
bi-Lipschitz on $B(\vec 0,4)\setminus U_\d$ (where it is regular),
$\det D\vec b_\d$ is bounded away from zero on each of these two
pieces, with constants that may depend on $\d$ (and hence on $\e$).
It remains to check its positivity in the transition layer $T_\d$.
There, the determinant estimates established in Step~2 give:
\begin{itemize}
\item In regions $b$, $d$, and $f$:
$\det D\vec b_\d\ge c\,s^2/r$.
Since $s\ge \d/2$ and $r\le 3$ on $T_\d$, we obtain
$\det D\vec b_\d\ge c\,\d^2$.
\item In region $e$:
$\det D\vec b_\d\ge c\,\det D\vec v
=c\,(1+\rho)^2\rho^{-2}\cos^3\!\varphi$.
On $T_\d\cap e$ one has $\dist(\vec x,U)=\rho\cos\varphi\ge \d/2$
and $\rho\le 1$, so $\cos\varphi\ge \d/(2\rho)\ge \d/2$
and $(1+\rho)^2/\rho^2\ge 1$. Hence
$\det D\vec b_\d\ge c\,\d^3$.
\end{itemize}
Combining these bounds with the positivity on $U_{\d/2}$ and
$B(\vec 0,4)\setminus U_\d$, we conclude that
\begin{equation}\label{eq:det-lower-global}
m(\d):=\essinf_{B(\vec 0,4)} \det D\vec b_\d \;>\;0 .
\end{equation}
The value of $m(\d)$ depends on $\d$ through both the transition-layer
estimates and the bi-Lipschitz constants of $\vec u_{\e(\d)}$.
Combining \eqref{eq:Dbdelta-Linfty}, \eqref{eq:inv-bound} and \eqref{eq:det-lower-global} yields
\begin{equation}\label{eq:Dbdelta-inv-Linfty}
\|(D\vec b_\delta)^{-1}\|_{L^\infty(B(\vec 0,4))}\le
\frac{c\,\|D\vec b_\delta\|_{L^\infty}^2}{m(\delta)} \;=\; C(\delta)<\infty .
\end{equation}

\smallskip
\noindent\emph{(iii) Injectivity and Lipschitz continuity of the inverse.}
We next argue that $\vec b_\delta$ is a homeomorphism.
On $U_{\delta/2}$ and on $B(\vec 0,4)\setminus U_\delta$ this follows from the fact that $\vec u$ and $\vec v$
are homeomorphisms. In the transition layer $T_\delta$ we treat each region separately:
\begin{itemize}
\item In region $b$, the map $\vec b_\delta$ has the form
\[
\vec b_\delta(\rho,\varphi,\theta)
= r_\delta(\rho-1)\,\vec d(\varphi,\theta)+z_\delta(\rho-1)\,\vec e_3,
\]
where
\[
r_\delta(s):=s+\chi_\delta(s)\sqrt2\,\e^\gamma,
\qquad
z_\delta(s):=\chi_\delta(s)\e^\gamma.
\]
The angular variables $(\varphi,\theta)$ are preserved, and the planar map
$s\mapsto (r_\delta(s),z_\delta(s))$ is injective on $[\delta/2,\delta]$ since
\[
r_\delta'(s)=1-\frac{2\sqrt2\,\e^\gamma}{\delta}>0
\]
whenever $c_0>2\sqrt2$.
\item In region $d$, the planar representative in $(r,x_3)$ is the composition of the bi-Lipschitz map
$(r,x_3)\mapsto (s,z)$ (planar representative of $\vec g$) with the explicit map $(s,z)\mapsto (w_r,w_3)$
given by \eqref{eq:wr-def}. The determinant lower bound in $d\cap T_\delta$ (see Step~2) prevents
degeneracy, and the choice \eqref{eq:delta-eps} (with $c_0$ large) ensures that the images of the two boundary
components $\{s=\delta/2\}$ and $\{s=\delta\}$ are separated (see Figure~\ref{fig:extrusion}), hence injectivity holds.
\item In region $e$, the angular variables $(\varphi,\theta)$ are again preserved, and the radial
profile $b_\rho(\rho,\varphi)$ is strictly increasing in $\rho$ because
\[
\p_\rho b_\rho\ge \cos\varphi\left(1-\frac{14\e^\gamma}{\delta}\right)>0
\]
by Step~2. Hence, for every $\varphi$, the map
$\rho\mapsto b_\rho(\rho,\varphi)$ is injective.
\item In region $f$, $\vec v$ is regular and $\vec u_{\e} \to \vec v$ in $C^1$; thus, for $\e$ sufficiently small,
the convex interpolation defining $\vec b_\delta$ remains injective (this is a standard stability property of
bi-Lipschitz maps under $C^1$ perturbations).
\end{itemize}
Since $\vec b_\delta$ is continuous and injective on the compact set $\overline{B(\vec 0,4)}$, it is a homeomorphism onto
its image.

Finally, because $\vec b_\delta$ is piecewise $C^1$ and $\det D\vec b_\delta>0$, the restriction of
$\vec b_\delta$ to each piece is a $C^1$ diffeomorphism. Hence the inverse is $C^1$ on each
corresponding image piece and satisfies
\[
D(\vec b_\delta^{-1})=(D\vec b_\delta)^{-1}\circ \vec b_\delta^{-1}
\]
there. Using \eqref{eq:Dbdelta-inv-Linfty}, we obtain a Lipschitz bound for $\vec b_\delta^{-1}$ on
each image piece. Since the image pieces cover $\vec b_\delta(B(\vec 0,4))$ and intersect only along
their boundaries (where the traces match), the inverse $\vec b_\delta^{-1}$ is globally Lipschitz on
$\vec b_\delta(B(\vec 0,4))$.

Therefore $\vec b_\delta$ is bi-Lipschitz.
\end{proof}


\subsection*{Acknowledgements} 

D.~Henao was supported by FONDECYT grant N.~1231401 and by
Center for Mathematical Modeling, FB210005, Basal ANID Chile.

C. Mora-Corral has been supported by the Spanish Agencia Estatal de Investigaci\'on through projects PID2024-158664NB-C2, PCI2024-155023-2 and CEX-2023-001347-S.

The research of R. Rodiac  has been supported by the French government, through the \(\textbf{UCA}^{JEDI}\) Investments in the Future project managed by the National Research Agency (ANR) with the reference number ANR-15-IDEX-01.

The hospitality of Vera \& Asociados during the preparation of the manuscript is gratefully acknowledged.

\bibliographystyle{abbrv}
\bibliography{BiblioLavrentiev}

\begin{thebibliography}{10}

\bibitem{AlKrMo24}
S.~Almi, S.~Kr{\"o}mer, and A.~Molchanova.
\newblock A new example for the {Lavrentiev} phenomenon in nonlinear
  elasticity.
\newblock {\em Z. Angew. Math. Phys.}, 75(1):21, 2024.
\newblock Id/No 2.

\bibitem{Ball1982}
J.~M. Ball.
\newblock Discontinuous equilibrium solutions and cavitation in nonlinear
  elasticity.
\newblock {\em Philos. Trans. Roy. Soc. London Ser. A}, 306(1496):557--611,
  1982.

\bibitem{Ball02}
J.~M. Ball.
\newblock Some open problems in elasticity.
\newblock In {\em Geometry, mechanics, and dynamics}, pages 3--59. Springer,
  New York, 2002.

\bibitem{BaMi85}
J.~M. Ball and V.~J. Mizel.
\newblock One-dimensional variational problems whose minimizers do not satisfy
  the {Euler}-{Lagrange} equation.
\newblock {\em Arch. Ration. Mech. Anal.}, 90:325--388, 1985.

\bibitem{BallMurat1984}
J.~M. Ball and F.~Murat.
\newblock {$W^{1,p}$}-quasiconvexity and variational problems for multiple
  integrals.
\newblock {\em J. Funct. Anal.}, 58(3):225--253, 1984.

\bibitem{BaHeMo17}
M.~Barchiesi, D.~Henao, and C.~Mora-Corral.
\newblock Local invertibility in {S}obolev spaces with applications to nematic
  elastomers and magnetoelasticity.
\newblock {\em Arch. Rational Mech. Anal.}, 224(2):743--816, 2017.

\bibitem{BaHEMoRo_23PUB}
M.~Barchiesi, D.~Henao, C.~Mora-Corral, and R.~Rodiac.
\newblock Harmonic dipoles and the relaxation of the neo-{H}ookean energy in
  3{D} elasticity.
\newblock {\em Arch. Ration. Mech. Anal.}, 247(4):Paper No. 70, 46, 2023.

\bibitem{BaHEMoRo_24PUB}
M.~Barchiesi, D.~Henao, C.~Mora-Corral, and R.~Rodiac.
\newblock On the lack of compactness in the axisymmetric neo-hookean model.
\newblock {\em Forum of Mathematics, Sigma}, 12:e26, 2024.

\bibitem{BaHEMoRo_24_SIAM}
M.~Barchiesi, D.~Henao, C.~Mora-Corral, and R.~Rodiac.
\newblock A relaxation approach to the minimization of the neo-{Hookean} energy
  in {3D}.
\newblock {\em SIAM J. Math. Anal.}, 56(6):7830--7845, 2024.

\bibitem{Bouchala_Hencl_Zhu_2024}
O.~Bouchala, S.~Hencl, and Z.~Zhu.
\newblock Weak limit of {{\(W^{1,2}\)}} homeomorphisms in {{\(\mathbb{R}^3\)}}
  can have any degree.
\newblock {\em Ann. Fenn. Math.}, 49(2):547--560, 2024.

\bibitem{Campbell_Dolezalova_Hencl_2025}
D.~Campbell, A.~Dole{\v{z}}alov{\'a}, and S.~Hencl.
\newblock Mission $p<n-1$: possible -- nonlinear elasticity beyond conventional
  limits.
\newblock Preprint, {arXiv}:2506.07543, 2025.

\bibitem{CiNe87}
P.~G. Ciarlet and J.~Ne{\v{c}}as.
\newblock Injectivity and self-contact in nonlinear elasticity.
\newblock {\em Arch. Rational Mech. Anal.}, 97(3):171--188, 1987.

\bibitem{CoDeLe03}
S.~Conti and C.~De~Lellis.
\newblock Some remarks on the theory of elasticity for compressible
  {N}eohookean materials.
\newblock {\em Ann. Sc. Norm. Super. Pisa Cl. Sci. (5)}, 2(3):521--549, 2003.

\bibitem{Dolezalova_Hencl_Maly_2023}
A.~Dole{\v{z}}alov{\'a}, S.~Hencl, and J.~Mal{\'y}.
\newblock Weak limit of homeomorphisms in {{\({W}^{1, n-1}\)}} and ({INV})
  condition.
\newblock {\em Arch. Ration. Mech. Anal.}, 247(5):54, 2023.
\newblock Id/No 80.

\bibitem{Dolezalova_Hencl_Molchanova_2024}
A.~Dole{\v{z}}alov{\'a}, S.~Hencl, and A.~Molchanova.
\newblock Weak limit of homeomorphisms in {{\(W^{1,n-1}\)}}: invertibility and
  lower semicontinuity of energy.
\newblock {\em ESAIM, Control Optim. Calc. Var.}, 30:32, 2024.
\newblock Id/No 37.

\bibitem{Evans_Gariepy_2015}
L.~C. Evans and R.~F. Gariepy.
\newblock {\em Measure theory and fine properties of functions}.
\newblock Textb. Math. Boca Raton, FL: CRC Press, revised ed. edition, 2015.

\bibitem{Federer69}
H.~Federer.
\newblock {\em Geometric measure theory}.
\newblock Springer, New York, 1969.

\bibitem{Foss03}
M.~Foss.
\newblock Examples of the {Lavrentiev} phenomenon with continuous {Sobolev}
  exponent dependence.
\newblock {\em J. Convex Anal.}, 10(2):445--464, 2003.

\bibitem{FoHrMi03a}
M.~Foss, W.~Hrusa, and V.~J. Mizel.
\newblock The {Lavrentiev} phenomenon in nonlinear elasticity.
\newblock {\em J. Elasticity}, 72(1-3):173--181, 2003.

\bibitem{FoHrMi03b}
M.~Foss, W.~J. Hrusa, and V.~J. Mizel.
\newblock The {Lavrentiev} gap phenomenon in nonlinear elasticity.
\newblock {\em Arch. Ration. Mech. Anal.}, 167(4):337--365, 2003.

\bibitem{GiMoSo98I}
M.~Giaquinta, G.~Modica, and J.~Sou\v{c}ek.
\newblock {\em {Cartesian currents in the calculus of variations. {I}}}.
\newblock Springer-Verlag, Berlin, 1998.

\bibitem{GiMoSo98II}
M.~Giaquinta, G.~Modica, and J.~Sou\v{c}ek.
\newblock {\em {Cartesian currents in the calculus of variations. {II}}}.
\newblock Springer-Verlag, Berlin, 1998.

\bibitem{Hardt_Lin_1986}
R.~Hardt and F.-H. Lin.
\newblock A remark on {$H^1$} mappings.
\newblock {\em Manuscripta Math.}, 56(1):1--10, 1986.

\bibitem{Hardt_Lin_1992}
R.~Hardt, F.-H. Lin, and C.-C. Poon.
\newblock Axially symmetric harmonic maps minimizing a relaxed energy.
\newblock {\em Comm. Pure Appl. Math.}, 45(4):417--459, 1992.

\bibitem{HeMo10}
D.~Henao and C.~Mora-Corral.
\newblock Invertibility and weak continuity of the determinant for the
  modelling of cavitation and fracture in nonlinear elasticity.
\newblock {\em Arch. Rational Mech. Anal}, 197:619--655, 2010.

\bibitem{HeMo11}
D.~Henao and C.~Mora-Corral.
\newblock Fracture surfaces and the regularity of inverses for {BV}
  deformations.
\newblock {\em Arch. Rational Mech. Anal.}, 201(2):575--629, 2011.

\bibitem{HeMo12}
D.~Henao and C.~Mora-Corral.
\newblock Lusin's condition and the distributional determinant for deformations
  with finite energy.
\newblock {\em Adv. Calc. Var.}, 5(4):355--409, 2012.

\bibitem{Kalayanamit_2025}
P.~Kalayanamit.
\newblock Sobolev regularity of the inverse for minimizers of the neo-{H}ookean
  energy satisfying condition {INV}.
\newblock {\em Proceedings of the Royal Society of Edinburgh: Section A
  Mathematics}, page 1–21, 2025.

\bibitem{Mazowiecka_Strzelecki_2017}
K.~Mazowiecka and P.~Strzelecki.
\newblock The {Lavrentiev} gap phenomenon for harmonic maps into spheres holds
  on a dense set of zero degree boundary data.
\newblock {\em Adv. Calc. Var.}, 10(3):303--314, 2017.

\bibitem{MuSp95}
S.~M{\"u}ller and S.~J. Spector.
\newblock An existence theory for nonlinear elasticity that allows for
  cavitation.
\newblock {\em Arch. Rational Mech. Anal.}, 131(1):1--66, 1995.

\end{thebibliography}

\end{document}